\documentclass{article}%
\usepackage{amsfonts}
\usepackage{amsmath}%
\usepackage{amssymb}%
\usepackage{graphicx}
\newtheorem{theorem}{Theorem}

\newtheorem{corollary}[theorem]{Corollary}

\newtheorem{definition}[theorem]{Definition}

\newtheorem{lemma}[theorem]{Lemma}

\newtheorem{proposition}[theorem]{Proposition}

\newenvironment{proof}[1][Proof]{\noindent\textbf{#1.} }{\ \rule{0.5em}{0.5em}}
\begin{document}

\title{A bracket polynomial for graphs. II.\\Links, Euler circuits and marked graphs }
\author{Lorenzo Traldi\\Lafayette College\\Easton, Pennsylvania 18042}
\date{}
\maketitle

\begin{abstract}
Let $D$ be an oriented classical or virtual link diagram with directed
universe $\vec{U}$. Let $C$ denote a set of directed Euler circuits, one in
each connected component of $U$. There is then an associated looped
interlacement graph $\mathcal{L}(D,C)$ whose construction involves very little
geometric information about the way $D$ is drawn in the plane; consequently
$\mathcal{L}(D,C)$ is different from other combinatorial structures associated
with classical link diagrams, like the checkerboard graph, which can be
difficult to extend to arbitrary virtual links. $\mathcal{L}(D,C)$ is
determined by three things: the structure of $\vec{U}$ as a 2-in, 2-out
digraph, the distinction between crossings that make a positive contribution
to the writhe and those that make a negative contribution, and the
relationship between $C$ and the directed circuits in $\vec{U}$ arising from
the link components; this relationship is indicated by marking the vertices
where $C$ does not follow the incident link component(s). We introduce a
bracket polynomial for arbitrary marked graphs, defined using either a formula
involving matrix nullities or a recursion involving the local complement and
pivot operations; the marked-graph bracket of $\mathcal{L}(D,C)$ is the same
as the Kauffman bracket of $D$. This provides a unified combinatorial
description of the Jones polynomial that applies seamlessly to both classical
and non-classical virtual links.

\bigskip


\bigskip


\end{abstract}

\section{Introduction}

A \textit{classical link} $L$ consists of finitely many, pairwise disjoint,
piecewise smooth closed curves in $\mathbb{R}^{3}$ or $\mathbb{S}^{3}$; each
individual closed curve is a \textit{component} of $L$. A \textit{regular
diagram} $D$ of $L$ is obtained from a regular projection in the plane -- that
is, a projection whose only singularities are double points called
\textit{crossings} -- by specifying the over-and under-crossing arcs at each
crossing. The projection itself is a 4-regular plane graph $U$, the
\textit{universe} of $D$. Saying that $U$ is a \textit{plane graph} means that
it is given with a specific embedding in $\mathbb{R}^{2}$; when the embedding
is forgotten $U$ becomes an \textit{abstract graph}. We presume $L$ is given
with orientations of its components; these orientations make $U$ into a 2-in,
2-out digraph $\vec{U}$. Combinatorial structures associated with $D$
incorporate geometric information about its embedding in the plane in various
ways -- for instance the classical checkerboard graph is defined using the
connected components of $\mathbb{R}^{2}-U$, a rigid graph specifies for each
vertex of $U$ the cyclic order of the incident edges, a signed Gauss code
specifies the cyclic order of the crossings on each link component along with
underpassing-overpassing information at each crossing, and an atom includes an
embedding of a graph on a surface.

\bigskip

Most combinatorial structures associated to classical link diagrams must be
adjusted when we try to extend them to \textit{virtual link diagrams}
\cite{Kv}. Each crossing in a virtual link diagram has three possible
configurations. In addition to the two classical configurations with different
choices of under- and over-crossing arcs, there is the virtual configuration,
in which two arcs of the diagram simply cross each other without intersecting.
The virtual crossings are \textquotedblleft not really there\textquotedblright%
\ \cite{Kv}, in the same way that a plane representation of a graph may
contain apparent edge-intersections that are \textquotedblleft not really
there.\textquotedblright\ The universe of a virtual link diagram is an
abstract 4-regular graph with vertices corresponding only to the classical
crossings; the virtual diagram gives a regular immersion of the graph in
$\mathbb{R}^{2}$, not an embedding as in the classical case. The virtual
crossings naturally suggest an embedding of the universe on a closed,
orientable surface $S$ of some genus, and the diagram can then be thought of
as representing a link in the thickened surface $S\times\mathbb{R}$.
Consequently one way to extend notions involving the plane geometry of
classical link diagrams is to use the geometry of the surface $S$; for
instance the classical checkerboard graphs extend in this manner to atoms
associated with virtual diagrams. (By the way we should make it clear that we
use the terms \textit{link}, \textit{virtual link} and \textit{classical or
virtual link} interchangeably; when we want to be restrictive we specify that
a link is classical or non-classical.)

\bigskip

The Kauffman bracket and Jones polynomial of a classical link $L$ may be
described in a variety of ways using combinatorial structures associated with
regular diagrams of $L$; Kauffman's basic description follows. Each crossing
of a regular diagram $D$ has two \textit{smoothings}, one denoted $A $ and the
other denoted $B$; distinguishing between them is another example of the use
of geometric information mentioned in the first paragraph above, as it
involves the orientation of the plane in which $D$ is drawn. If $D$ has $n$
crossings then it has $2^{n}$ \textit{states}, obtained by applying either the
$A$ or the $B$ smoothing at each crossing. Given a state $S$ let $a(S)$ denote
the number of $A$ smoothings in $S$, $b(S)=n-a(S)$ the number of $B$
smoothings in $S$, and $c(S)$ the number of simple closed curves in $S $,
including any crossing-free components that might appear in $D$. (These simple
closed curves are called ``loops'' in \cite{Kv}, but we reserve this word for
graph-theoretic use.) Then the \textit{Kauffman bracket polynomial} of $D$ is
a sum indexed by states:
\[
\lbrack D]=\sum_{S}A^{a(S)}B^{b(S)}d^{c(S)-1}.
\]
%

\begin{figure}
[ptb]
\begin{center}
\includegraphics[
trim=0.360464in 7.634094in 0.000000in 1.133077in,
height=1.0897in,
width=4.363in
]%
{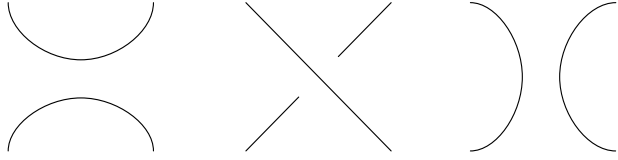}%
\caption{The $A$ smoothing is on the left and the $B$ smoothing is on the
right.}%
\label{wifig1}%
\end{center}
\end{figure}

\bigskip

Almost immediately after the Kauffman bracket was introduced, it was related
to the Tutte polynomial through the checkerboard graph construction \cite{Th}.
A classical link diagram $D$ has an \textit{alternating} orientation on its
universe $U$ -- i.e., the edges of $U$ may be directed in such a way that when
one traverses any component of the link, one encounters forward- and
backward-directed edges alternately.\ Such an alternating orientation arises
naturally from the cycles that bound the complementary regions of $D$, and is
associated with the relationship between Kauffman states of $D$ and spanning
forests in the checkerboard graph that underlies the Jones-Tutte connection.
This connection between spanning forests and Kauffman states is also part of a
combinatorial discussion of the Alexander polynomial \cite{FKT}. By the way,
the relationship between Eulerian circuits and spanning trees in regular
directed graphs, known as the BEST theorem, appeared in the combinatorial
literature long before the Jones polynomial was introduced \cite{BA, St}; the
special relationship between circuit partitions in a 4-regular plane graph and
the Tutte polynomial of an associated checkerboard graph also appeared first
outside knot theory \cite{J, Ma}.

\bigskip

The Kauffman bracket of a virtual link diagram is described in almost exactly
the same way as that of a classical link diagram; the only difference is that
we cannot say that $c(S)$ denotes the number of \textit{simple }closed curves
in $S$, as the closed curves need not be simple. Much of the Jones-Tutte
machinery extends directly to virtual link diagrams which possess alternating
orientations or (equivalently) orientable atoms; see \cite{Ka} and also
Chapter IV of \cite{MK}. If the Tutte polynomial is augmented with topological
information to give an invariant of graphs embedded on surfaces as in
\cite{BR1, BR2}, then the relationship between links in thickened surfaces and
virtual links leads to a general relationship between this topological Tutte
polynomial and the Jones polynomial \cite{CP, CV}.

\bigskip

In \cite{TZ} Zulli and the present author proposed a different, purely
combinatorial approach to the Jones polynomials and Kauffman brackets of
classical and virtual knots (one-component links) that uses surprisingly
little geometric information about knot diagrams, and does not involve either
the ordinary Tutte polynomial or the topological Tutte polynomial. If $D$ is a
regular diagram of a classical or virtual knot $K$ then the universe $U$ of
$D$ has an interlacement graph defined with respect to the Euler circuit
obtained directly from $K$, and the Kauffman bracket $[D]$ can be obtained
from this interlacement graph without using complementary regions, ordering
the crossings, specifying over- and under-crossing arcs, ordering the edges
incident on a vertex, distinguishing between smoothings, counting closed
curves, or considering closed surfaces of any genus. All that is required is
to record whether each crossing's contribution to the writhe is +1 or -1; this
is done by attaching a loop to each vertex that corresponds to a crossing
whose writhe contribution is -1. The resulting \textit{looped interlacement
graph} is denoted $\mathcal{L}(D)$. $\mathcal{L}(D)$ is obtained by
considering $U$ as an abstract graph rather than an embedded one, so any
virtual crossings are really not there (rather than merely \textquotedblleft
not really there\textquotedblright)\ as far as $\mathcal{L}(D)$ is concerned.

\bigskip

The analysis of \cite{TZ} describes the Jones polynomial of a classical or
virtual knot as the result of the composition of three functions. The first
function is the looped interlacement graph construction. The second function
associates a 3-variable graph bracket polynomial to an arbitrary graph, using
either a recursion involving the local complement and pivot operations that
have appeared in the theory of circle graphs and interlace polynomials
\cite{A1, A2, A, Bc, B, K} or a formula involving matrix nullities over
$GF(2)$. Zulli \cite{Z} proved that this formula, when applied to looped
interlacement graphs, correctly assesses the number of closed curves in each
Kauffman state. The third function modifies the 3-variable graph bracket
polynomial so that the resulting graph Jones polynomial is invariant under
appropriate graph\ Reidemeister moves.

\bigskip

\bigskip

$D\qquad\mapsto\qquad\mathcal{L}(D)\qquad\qquad\mapsto\qquad\qquad
\lbrack\mathcal{L}(D)]\qquad\mapsto\qquad V_{\mathcal{L}(D)}(t)$

\bigskip

\noindent(virtual)\qquad\ \ looped interlace-\qquad\qquad3-variable
graph\qquad\ \ \ \ \ \ graph

\noindent\ \ \ knot\qquad\ \ \ ment graph using \qquad\qquad\ \ \ \ bracket
\qquad\qquad\ \ \ \ \ Jones

\noindent diagram\qquad\ \ Euler circuit $K\qquad\qquad\ \ \ \ $%
polynomial\qquad\qquad\ polynomial

\bigskip

\bigskip

The purpose of the present paper is to present an extension of this
description to the Jones polynomials of classical and virtual links. Each of
the three steps must be modified to accommodate the extension.

\bigskip

\bigskip

$D\qquad\mapsto\qquad\mathcal{L}(D,C)\qquad\qquad\mapsto\quad\qquad
\lbrack\mathcal{L}(D,C)]\qquad\mapsto\qquad V_{\mathcal{L}(D,C)}(t)$

\bigskip

\noindent(virtual)\qquad\ \ \ looped interlace-\qquad\qquad\qquad
3-variable\qquad\ \ \ \ \ \ marked-graph

\noindent\ \ \ link\qquad\qquad ment graph using \qquad
\ \ \ \ \ \ \ \ marked-graph \qquad\ \qquad Jones

\noindent diagram\qquad\ \ \ an Euler system $C\qquad\ \ \ \ $%
bracket\ polynomial\qquad\ polynomial

\bigskip

The first step is again the construction of a looped interlacement graph. An
interlacement graph of a 4-regular graph or a 2-in, 2-out digraph is
constructed by using a specific Euler circuit, or if the original graph is not
connected a specific \textit{Euler system} containing an Euler circuit for
each connected component. A classical or virtual knot diagram has an Euler
circuit arising naturally from the diagrammed knot, but in general there is no
canonical choice of an Euler system in a universe $U$ associated to a diagram
$D$ representing a multi-component link. Consequently our discussion of this
step in Section 2 involves two complications that did not affect \cite{TZ}:
$U$ need not be connected, and we must allow a variety of Euler systems in $U$.

\bigskip

The first of these complications is accommodated using a common
graph-theoretic convention, which allows graphs to include free loops; a free
loop is not incident on any vertex, but does count as a connected component.
If $U$ is a disconnected graph with $c(U)>1$ connected components, we
\textquotedblleft record\textquotedblright\ this fact by inserting $c(U)-1$
free loops into the interlacement graph $\mathcal{L}(D,C)$. This might seem
unnatural at first, but it is necessary if we want to recover the Kauffman
bracket from $\mathcal{L}(D,C)$; without some such convention interlacement
cannot detect crossing-free unknotted components, or (more generally)
distinguish connected sums from split unions.

\bigskip%

\begin{figure}
[ptb]
\begin{center}
\includegraphics[
trim=0.799862in 6.689842in 0.000000in 1.072899in,
height=2.2295in,
width=5.61in
]%
{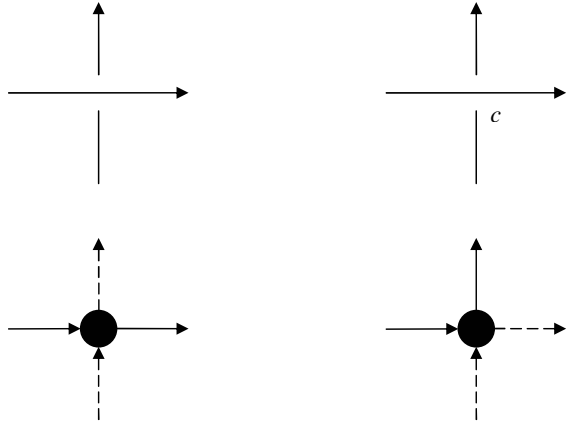}%
\caption{Marked crossings indicate the transitions of a directed Euler
system.}%
\label{lbracfi2}%
\end{center}
\end{figure}

The second complication is handled by adopting a convention to identify the
Euler system used to produce a looped interlacement graph $\mathcal{L}(D,C)$.
We presume the link $L$ is given with orientations of its components, and we
restrict our attention to directed Euler systems of the directed universe
$\vec{U}$. Each crossing of $D$ is then marked to record which of the two
possible \textit{transitions} (edge-pairings) occurs at that crossing in the
Euler system $C$ used to produce $\mathcal{L}(D,C)$. (Transitions provide a
standard combinatorial description of circuits in 4-regular graphs and 2-in,
2-out digraphs; see for instance \cite{E, J, K}.) Crossings at which $C$
follows the incident link component(s) are left unmarked, and crossings at
which $C$ is orientation-consistent with the incident link component(s) but
does not\ follow it (or them) are marked with the letter $c$; see Figures
\ref{lbracfi2} and \ref{lbracfi3}. The same convention is used to mark the
vertices of $U$ and $\mathcal{L}(D,C)$. According to results discussed in
\cite{A1, A2, K, Pev, U}, different directed Euler systems of $\vec{U}$ are
interconnected through \textit{transposition}, a combinatorial operation. It
follows that changing the choice of $C$ affects $\mathcal{L}(D,C)$ through an
appropriately modified version of the pivot operation, which we call a
\textit{marked pivot}.%

\begin{figure}
[tb]
\begin{center}
\includegraphics[
trim=0.796564in 4.573995in 0.000000in 0.910306in,
height=3.7801in,
width=5.3904in
]%
{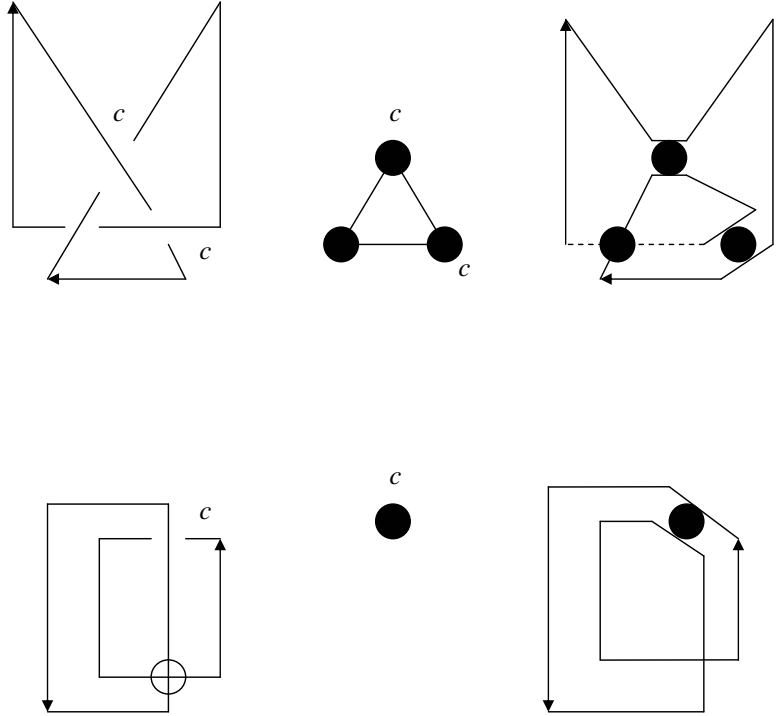}%
\caption{Marked diagrams and interlacement graphs indicate Euler systems.}%
\label{lbracfi3}%
\end{center}
\end{figure}

\bigskip

As in \cite{TZ}, the second step uses a formula involving matrix nullities
over $GF(2)$ to associate a 3-variable bracket polynomial to an arbitrary
graph; but now we consider arbitrary graphs with marked vertices. The
marked-graph bracket polynomial of a looped interlacement graph $\mathcal{L}%
(D,C)$ is the same as the Kauffman bracket polynomial of $D$; the relationship
between the matrix nullity formula and the Kauffman bracket is verified using
an extension of a combinatorial equality due to Cohn and Lempel \cite{CL}.
This equality is presented briefly in Section 4, and more fully in \cite{T};
it incorporates the Cohn-Lempel equality along with results used in \cite{L,
Me, Z}.

\bigskip

In Section 5 we define the 3-variable marked-graph bracket polynomial and
deduce that $[\mathcal{L}(D,C)]=[D]$. We also prove a very useful result: the
bracket is invariant under the marked pivot operation mentioned above. This
invariance allows us to develop the theory of the marked-graph bracket using
many of the results of \cite{TZ}, even though marked graphs did not appear
there; in essence, we often use marked pivots to simply push marked vertices
out of the way. For instance, in Section 6 we show that the recursion given in
\cite{TZ} extends to a recursive description of the marked-graph bracket
polynomial involving the local complementation and pivot operations. The
recursion uses formulas that can only be applied to special configurations in
which certain vertices have no marked neighbors; such configurations can
always be created using marked pivots.

\bigskip

The third step involves proving that the marked-graph bracket polynomial
yields a marked-graph Jones polynomial that is invariant under certain
Reidemeister moves on marked graphs; these graph moves generate the
``images''\ of the familiar Reidemeister moves of knot theory under the
construction of $\mathcal{L}(D,C)$. We say \textit{generate} because we do not
explicitly exhibit all the marked-graph Reidemeister moves; following an idea
of \"{O}stlund \cite{O}, we only exhibit enough basic moves to generate the
rest through composition. As in \cite{TZ}, the fact that $\mathcal{L}(D,C)$ is
defined by considering $U$ as an abstract graph rather than a graph drawn on a
plane or a surface is reflected in the fact that we can simply ignore the
virtual Reidemeister moves of \cite{Kv}.

\bigskip

In sum, our results show that the Kauffman bracket and Jones polynomial can be
defined for either of the following abstract (non-embedded) combinatorial structures.

\bigskip

\begin{itemize}
\item a 2-in, 2-out digraph $\vec{U}$, given with a partition of its edge-set
into directed circuits and a partition of its vertex-set into two subsets

\item an undirected graph $G$, given with a partition of its vertex-set into
two subsets
\end{itemize}

\bigskip

It would be interesting to understand how these two types of abstract
combinatorial structures are related to other knot-theoretic notions. For
instance, orientable and non-orientable atoms have been used for a variety of
purposes, e.g., to construct Khovanov homology for virtual knots \cite{Mv}.
Atoms involve graphs embedded on surfaces, so they contain more geometric
information than the abstract combinatorial structures. Perhaps some
constructions based on atoms do not require all of this geometric information,
and can be modified to use one of the abstract combinatorial structures.
Manturov has also recently introduced the idea of \textit{free} knots and
links, which are equivalence classes of 4-regular graphs under modified
versions of the Reidemeister moves \cite{M1, M2}; they are related to the
first kind of abstract combinatorial structure mentioned above, with the
vertex-partition forgotten. Perhaps the combinatorial theory of circuit
partitions in 4-regular graphs will turn out to be useful in defining
invariants of these objects.

\bigskip

It would also be interesting to conduct a purely combinatorial investigation
of the bracket as an invariant of either of these structures, apart from the
special case in which $\vec{U}$ is the universe of a diagram of an oriented
link, the circuit partition of $E(U)$ consists of the link's components, the
partition of $V(U)$ represents the distinction between positive and negative
crossings, $G$ is the looped interlacement graph of $\vec{U}$ with respect to
an Euler system $C$, the partition of $V(G)$ represents vestigial information
regarding $C$, and free loops represent all but one of the connected
components of $U$. In particular, we wonder whether there might be a useful
relationship between the bracket polynomial and the interlace polynomials of
\cite{A1, A2, A} or the multivariate interlace polynomial of \cite{Ci}. The
same comments apply to the Jones polynomial, which is naturally regarded as a
simplification of the bracket, modified so as to be invariant under a certain
equivalence relation which is motivated topologically but may be described in
purely combinatorial terms.

\bigskip

Before closing this introduction we should thank L. Zulli for many hours spent
working together to formulate and make precise the ideas of \cite{TZ}, and
also for his comments on early drafts of the current paper. D. P. Ilyutko was
kind enough to bring \cite{IM} to our attention; the present work developed as
we tried to understand the differences and similarities of the two approaches.
V. O. Manturov's good advice improved the paper in several ways. We are also
grateful to Lafayette College for its support.

\section{Euler circuits and interlacement graphs}

In this section we discuss some terminology and results about Eulerian graphs,
and then apply these notions to link diagrams. Eulerian graphs have been
studied for centuries; it is not surprising that terminology has varied over
the years, and some results have been rediscovered. For us a \textit{graph}
may have loops or multiple edges, and each edge has two distinct
\textit{directions}; in a \textit{directed graph} a preferred direction has
been chosen for each edge. The terms \textit{adjacent} and \textit{neighbor}
refer to non-loop edges only; no vertex is adjacent to itself or a neighbor of
itself. In order to provide appropriate notation for the two different
directions of a loop, each edge is technically regarded as consisting of two
distinct half-edges, one incident on each end-vertex; the preferred direction
of a directed edge is expressed by designating one half-edge
as\textit{\ initial }and the other as \textit{terminal}. We will almost always
abuse notation by leaving it to the reader to split edges into half-edges. A
\textit{path} in a graph is a sequence $v_{1},e_{1},v_{2},e_{2},...,v_{n-1}%
,e_{n-1},v_{n}$ (technically $v_{1},h_{1},h_{1}^{\prime},v_{2},h_{2}%
,h_{2}^{\prime},...,v_{n-1},h_{n-1},h_{n-1}^{\prime},v_{n}$ where $e_{i}$ has
the half-edges $h_{i}$ and $h_{i}^{\prime}$) such that for each $i<n$, the
vertices $v_{i}$ and $v_{i+1}$ are incident on the edge $e_{i}$; in a directed
graph a \textit{directed path} must respect all edge-directions, and an
\textit{undirected path} might not respect some edge-directions. If
$v_{1}=v_{n}$ the path is \textit{closed}; in this case it is customary but
not mandatory to omit $v_{n}$ from the list. A \textit{circuit} is a closed
path in which no edge appears more than once; digraphs have \textit{directed
circuits} and \textit{undirected circuits}. An \textit{Euler circuit} in a
connected graph is a circuit in which every edge of the graph appears, and an
\textit{Euler system} in a graph is a\ set that contains one Euler circuit for
each connected component. A connected, undirected graph has an Euler circuit
if and only if every vertex is of even degree; a connected, directed graph has
a directed Euler circuit if and only if every vertex has equal indegree and outdegree.

\bigskip

We allow graphs to have \textit{free loops}. These do not interact in any way
with the vertices and edges of the graph, and they do not affect anything we
would otherwise say about the graph, except for the fact that each free loop
is counted as a connected component. In particular, a free loop is not a kind
of loop; indeed if the terminology were not standard we would call them
``empty components.''\ For example, if an undirected graph $G $ consists of a
4-cycle and two free loops then $G$ is 2-regular, and if $e_{1},e_{2}$ are two
edges in $G$ that are not incident on the same vertex then $G-e_{1}-e_{2}$ is
bipartite and 1-regular; $G$ has three connected components and $G-e_{1}%
-e_{2}$ has four. As a free loop is a connected component, an Euler system for
a graph must contain all the graph's free loops. In figures representing
graphs, free loops are drawn as circles with no incident vertices.

\bigskip

Let $U$ be an undirected 4-regular graph with $V(U)=\{v_{1},...,v_{n}\}$. It
need not be the universe of any link diagram; we use the letter $U$ only for
notational consistency.

\begin{definition}
\label{interlacement} \emph{\cite{RR}} If $C$ is an Euler system for $U$ then
the \emph{interlacement matrix} $I(U,C)$ is the $n\times n$ matrix over
$GF(2)$ whose $ij$ entry is $\emph{1}$ if and only if $i\neq j$ and $v_{i}$
and $v_{j}$ are \emph{interlaced} in $C$, i.e., when we follow $C$ starting at
$v_{i}$ we encounter $v_{j}$, then $v_{i}$, then $v_{j}$ again before finally
returning to $v_{i}$.
\end{definition}

\begin{definition}
\label{intergraph} If $U$ has $c(U)$ connected components then the
\emph{interlacement graph} $\mathcal{I}(U,C)$ is obtained by adjoining
$c(U)-1$ free loops to the simple graph with vertex-set $V(U)$ and adjacency
matrix $I(U,C)$.
\end{definition}

Definition \ref{intergraph} differs from the corresponding definition of
\cite{RR} in its use of free loops. Observe that if $U$ is disconnected then
$I(U,C)$ consists of diagonal blocks corresponding to the interlacement
matrices of the connected components of $U$. The same occurs if $U$ is a
``connected sum,''\ i.e., if $U$ has two subgraphs $U_{1}$ and $U_{2}$ such
that $V(U)=V(U_{1})\cup V(U_{2})$ and $\left|  E(U)-E(U_{1})-E(U_{2})\right|
=2$. The free loops in Definition \ref{intergraph} allow $\mathcal{I}(U,C)$ to
distinguish these two situations: the interlacement graph of the disjoint
union of $U_{1}$ and $U_{2}$ has one more free loop than the interlacement
graph of a connected sum of $U_{1}$ and $U_{2}$.

\begin{definition}
\emph{\cite{K}} The $\kappa$\emph{-transform} $C\ast v$ of an\ Euler system
$C$ at a vertex $v$ is obtained by reversing one of the two $v$-to-$v$ paths
within the Euler circuit of $C$ in the connected component of $U$ containing
$v$.
\end{definition}

The interlacement graph $\mathcal{I}(U,C\ast v)$ is obtained from
$\mathcal{I}(U,C)$ by toggling (i.e., changing) every adjacency involving
neighbors of $v$ in $\mathcal{I}(U,C)$; that is, if $x,y\not \in \{v\}$ are
two distinct neighbors of $v$ in $\mathcal{I}(U,C)$, then $x$ and $y$ are
adjacent in $\mathcal{I}(U,C\ast v)$ if and only if they are not adjacent in
$\mathcal{I}(U,C)$.

\bigskip

If we apply a single $\kappa$-transformation to a directed Euler system of a
2-in, 2-out digraph $\vec{U}$ then the result is no longer compatible with the
edge-directions of $\vec{U}$. Suppose $C$ is a directed Euler system in
$\vec{U}$, and suppose $v$ and $w$ are two vertices of $\vec{U}$ that are
interlaced with respect to $C$. Then $C$ contains a circuit $vC_{1}%
wC_{2}vC_{3}wC_{4}$. The $\kappa$-transform $C\ast v$ contains an undirected
circuit $vC_{1}wC_{2}v\bar{C}_{4}w\bar{C}_{3}$; the overbars indicate that the
paths $wC_{4}v$ and $vC_{3}w$ have been reversed. Then $C\ast v\ast w$
contains $vC_{1}wC_{4}v\bar{C}_{2}w\bar{C}_{3}$ and $C\ast v\ast w\ast v$
contains $vC_{1}wC_{4}vC_{3}wC_{2}$. The lack of overbars indicates that this
last is a directed circuit of $\vec{U}$. Arratia, Bollob\'{a}s and Sorkin
\cite{A1, A2} call the operation $C\mapsto C\ast v\ast w\ast v$ a
\textit{transposition} on the pair $vw$. Comparing $vC_{1}wC_{2}vC_{3}wC_{4}$
to $vC_{1}wC_{4}vC_{3}wC_{2}$, we see that $\mathcal{I}(U,C\ast v\ast w\ast
v)$ is obtained from $\mathcal{I}(U,C)$ by interchanging the neighbors of $v$
and $w$, and toggling every adjacency between vertices $x,y\not \in \{v,w\}$
such that $x$ is adjacent to $v$, $y$ is adjacent to $w$, and either $x$ is
not adjacent to $w$ or $y$ is not adjacent to $v$. That is, $\mathcal{I}%
(U,C\ast v\ast w\ast v)$ is obtained from the \textit{pivot} $\mathcal{I}%
(U,C)^{vw}$ by interchanging the neighbors of $v$ and $w$. (Simply
interchanging the names of $v$ and $w$ would have the same effect, but we
prefer not to do this as it would complicate our discussion later; see for
instance Lemma \ref{transform} below.)

\bigskip

A natural way to identify an Euler system $C$ in a 4-regular graph $U$ is to
record, for each vertex of $U$, which of the three possible
\textit{transitions} (matchings of the incident half-edges) is used by $C$.
This gives a rather complicated description, listing a pair of incident
half-edges at each vertex. It is much simpler to use transitions to compare
two Euler systems $C$ and $C^{\prime}$. If we choose orientations for the
Euler circuits in $C$ then these orientations allow us to make $U$ into a
2-in, 2-out digraph $\vec{U}$: at each vertex $v$, $C^{\prime}$ must have
either the same transition as $C$, the other transition that is compatible
with the edge-directions in $\vec{U}$, or the transition that is incompatible
with the edge-directions of $\vec{U}$. We say $C^{\prime}$ \textit{follows}
$C$ at $v$, $C^{\prime}$ is \textit{compatible with }$C$ \textit{without
following it} at $v$, or $C^{\prime}$ \textit{is incompatible} with $C$ at
$v$, respectively, to describe the three cases. The same description of the
relationship between the transitions of $C$ and $C^{\prime}$ at $v$ will hold
if some circuits of $C$ are oriented differently. This simple description of
the relationship between $C$ and $C^{\prime}$ is useful in proving the
following result, due to Kotzig \cite{K}, Ukkonen \cite{U} and Pevzner
\cite{Pev}.

\begin{proposition}
\label{ktransform} Every two Euler systems $C$ and $C^{\prime}$ of a 4-regular
undirected graph $U$ are connected to each other through a sequence of
$\kappa$-transformations. If there is no vertex at which $C^{\prime}$ is
incompatible with $C$, they are connected through a sequence of transpositions.
\end{proposition}

\begin{proof}
If $C$ and $C^{\prime}$ have the same transition at every vertex, then
$C=C^{\prime}$. Proceeding by induction on the number of vertices at which
their transitions differ, we have two cases.

If their transitions are incompatible at a vertex $v$ then let $C^{\prime
\prime}=C\ast v$. $C^{\prime\prime}$ has the same transitions as $C$ except at
$v$, where it shares the transition of $C^{\prime}$, so the inductive
hypothesis gives the result.

Suppose instead that there is no vertex at which $C$ and $C^{\prime}$ have
incompatible transitions. Let $W$ denote the subset of $V(U)$ containing the
vertices at which the transitions of $C$ and $C^{\prime}$ differ, and let
$w\in W$ have the shortest possible $w$-to-$w$ path within $C$; let the Euler
circuit appearing in $C$ that contains $w$ be $wC_{1}wC_{2}$, with $C_{1}$ no
longer than $C_{2}$. If no vertex of $W$ appears on $C_{1}$ the transitions of
$C^{\prime}$ are the same as those of $C$ at every vertex of $wC_{1}$ except
$w$; then the circuit $wC_{1}$ appears in $C^{\prime}$, an impossibility as
$wC_{1}$ is not an Euler circuit. The choice of $w$ implies that no vertex of
$W$ can appear twice on $C_{1}$; hence there is a vertex $w^{\prime}\in W$
that appears precisely once on $C_{1}$. That is, $w$ and $w^{\prime}$ are
interlaced with respect to $C$. Let $C^{\prime\prime}=C\ast w\ast w^{\prime
}\ast w$. Then $C^{\prime\prime}$ has the same transitions as $C$ at all
vertices except $w$ and $w^{\prime}$, where it is compatible with $C$; hence
$C^{\prime\prime}$ has the same transitions at $w$ and $w^{\prime}$ as
$C^{\prime}$, has different transitions from $C^{\prime}$ at two fewer
vertices than $C$, and is not incompatible with $C^{\prime}$ at any vertex.
The inductive hypothesis gives the result.
\end{proof}

\bigskip

Just as using transitions to describe the relationship between two\ Euler
systems is simpler than using transitions to describe a single Euler system,
it is easier to describe an Euler system in the universe of a link diagram
than it is to describe an Euler system in an arbitrary 4-regular graph.
Suppose $D$ is a diagram of a classical or virtual link $L$, given with
orientations of its components; these orientations make $U$ into a 2-in, 2-out
digraph $\vec{U}$. (Any link component that is crossing-free in $D$ appears as
a free loop in $U$.) At each crossing in $D$, there are three transitions: one
followed by the component(s) of $L$, one that does not follow the components
of $L$ but is consistent with the edge-directions of $\vec{U}$, and a third
that is inconsistent with these edge-directions. A directed Euler system for
$\vec{U}$ will not involve any transitions of the third type. A simple
description of a directed Euler system $C$ is incorporated into a link diagram
$D$ by using the letter $c$ to mark each crossing at which $C$ is
orientation-consistent with $L$ without following it. (Crossings at which $C$
follows $L$ are left unmarked.) See Figures \ref{lbracfi2} and \ref{lbracfi3}
in the Introduction. We do not use the traditional notation for smoothings
(with $0$ instead of $c$), in order to avoid confusion between notation that
identifies an Euler system in $U$ and notation that identifies a new diagram
obtained by modifying $D$. The same system is used to mark the vertices in the
interlacement graph $\mathcal{I}(U,C)$, resulting in the \textit{marked
interlacement graph }$\mathcal{I}(D,C)$.

\bigskip%

\begin{figure}
[h]
\begin{center}
\includegraphics[
trim=1.338326in 6.878107in 0.000000in 1.190565in,
height=1.9995in,
width=5.2062in
]%
{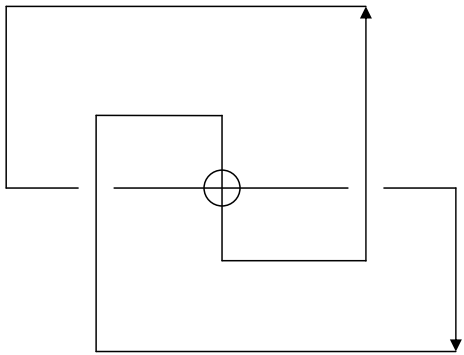}%
\caption{Kamada's virtual knot.}%
\label{lbracfi4}%
\end{center}
\end{figure}

As noted in the Introduction, the famous Jones-Tutte connection of \cite{Th}
is related to a choice of edge-directions in $U$ which is \textit{alternating}
in the sense that as one traverses a component of the link, one encounters a
reversal in edge-direction each time one passes through a vertex. If $U$
admits an alternating orientation as a 2-in, 2-out digraph then there is an
alternating diagram $D^{alt}$ with universe $U$: at each classical crossing
the arc with both edges directed inward is the underpassing arc in $D^{alt}$.
Lemma 7 of \cite{Ka} then tells us that $L$ is checkerboard colorable;
\cite{Ka} also tells us that not all virtual links are checkerboard colorable,
so restricting attention to Euler systems that give alternating orientations
would involve a significant loss of generality. One way to avoid this loss of
generality is to use nonorientable atoms; see \cite{Mv} for instance. Our
approach avoids loss of generality in a different way, by using
edge-directions that are consistent with the orientations of the components of
$L$ instead of using alternating edge-directions. By the way, a simple
induction on the number of virtual crossings shows that it is always possible
to choose an alternating orientation for a virtual knot diagram, but the
orientation may not correspond to a 2-in, 2-out digraph. For instance every
alternating orientation of the virtual knot diagram from \cite{Ka} indicated
in Figure \ref{lbracfi4} produces a vertex of indegree 4 and a vertex of
outdegree 4. Diagrams of multi-component virtual links need not have any
alternating orientations at all; for instance consider a diagram with only one
virtual crossing, which involves two different link components.

\begin{definition}
Suppose we are given a diagram $D$ of an oriented classical or virtual link
$L$, along with a directed Euler system $C$ for $\vec{U}\,$. Then the
\emph{looped interlacement graph}\textit{\ }$\mathcal{L}(D,C)$ is obtained
from $\mathcal{I}(D,C)$ by attaching a loop at each vertex corresponding to a
negative crossing.
\end{definition}%

\begin{figure}
[h]
\begin{center}
\includegraphics[
trim=0.000000in 7.690801in 0.000000in 1.338508in,
height=1.0309in,
width=4.9753in
]%
{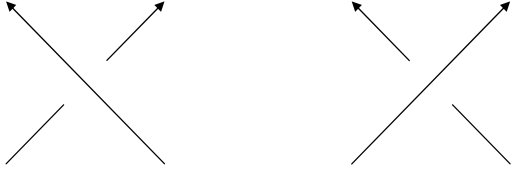}%
\caption{The crossing on the left is negative.}%
\label{lbracfi5}%
\end{center}
\end{figure}

\bigskip

Recall that if $v$ and $w$ are vertices of a graph $G$ then the \textit{pivot}
$G^{vw}$ is the graph obtained from $G$ by toggling (reversing) the adjacency
between every pair of vertices $x,y\not \in \{v,w\}$ such that $x$ is a
neighbor of $v$, $y$ is a neighbor of $w$, and either $x$ is not a neighbor of
$w$ or $y$ is not a neighbor of $x$. (Pivots have no effect on loops.) We
observed above that if a 4-regular graph $U$ has vertices $v,w$ that are
interlaced with respect to an Euler system $C$, then $\mathcal{I}(U,C\ast
v\ast w\ast v)$ is the graph obtained from $\mathcal{I}(U,C)^{vw}$ by
interchanging the neighbors of $v$ and $w$. This observation applies to looped
interlacement graphs as follows.

\begin{lemma}
\label{transform} Suppose $D$ is an oriented link diagram with a directed
Euler system $C$, and $v,w$ are adjacent in the looped interlacement graph
$\mathcal{L}(D,C)$. Then $\mathcal{L}(D,C\ast v\ast w\ast v)$ is the marked
graph obtained from $\mathcal{L}(D,C)^{vw}$ by toggling the marks on $v$ and
$w$ (i.e., removing a $c$ at $v$ or $w$ if one is there, and adjoining a $c$
at $v$ or $w$ if one is not there), and interchanging the neighbors of $v$ and
$w$.
\end{lemma}

\begin{proof}
At any vertex other than $v$ or $w$, $C\ast v\ast w\ast v$ has the same
transition as $C$. At $v$ and $w$, $C\ast v\ast w\ast v$ is compatible with
$C$ without following it.
\end{proof}

\section{Examples}

Figure \ref{lbracfk1} illustrates some of the definitions given in\ Section 2.
Each row of the figure has an oriented knot diagram on the left-hand side, the
corresponding directed universe in the middle, and the resulting looped
interlacement graph on the right. The first two examples are different marked
versions of a diagram of the figure-eight knot; the transitions of the Euler
circuits are indicated by the patterns of dashes in the corresponding
universes. The third example is the well-known virtual knot of \cite{KS}, and
the fourth is a virtual unknot diagram; their directed universes are
isomorphic abstract graphs. All Euler circuits in the third and fourth
directed universes give rise to the same looped interlacement graph, so we
have not bothered to mark the diagrams or indicate Euler circuits.

\bigskip

Note that Definition \ref{intergraph} does not require any free loops because
all four diagrams are connected. If we were to consider instead the split
diagram consisting of the entire left-hand side of the figure, then the looped
interlacement graph would include the right-hand side together with three
additional free loops.%

\begin{figure}
[ptb]
\begin{center}
\includegraphics[
trim=1.069843in 1.605995in 0.536984in 0.000000in,
height=6.8459in,
width=5.009in
]%
{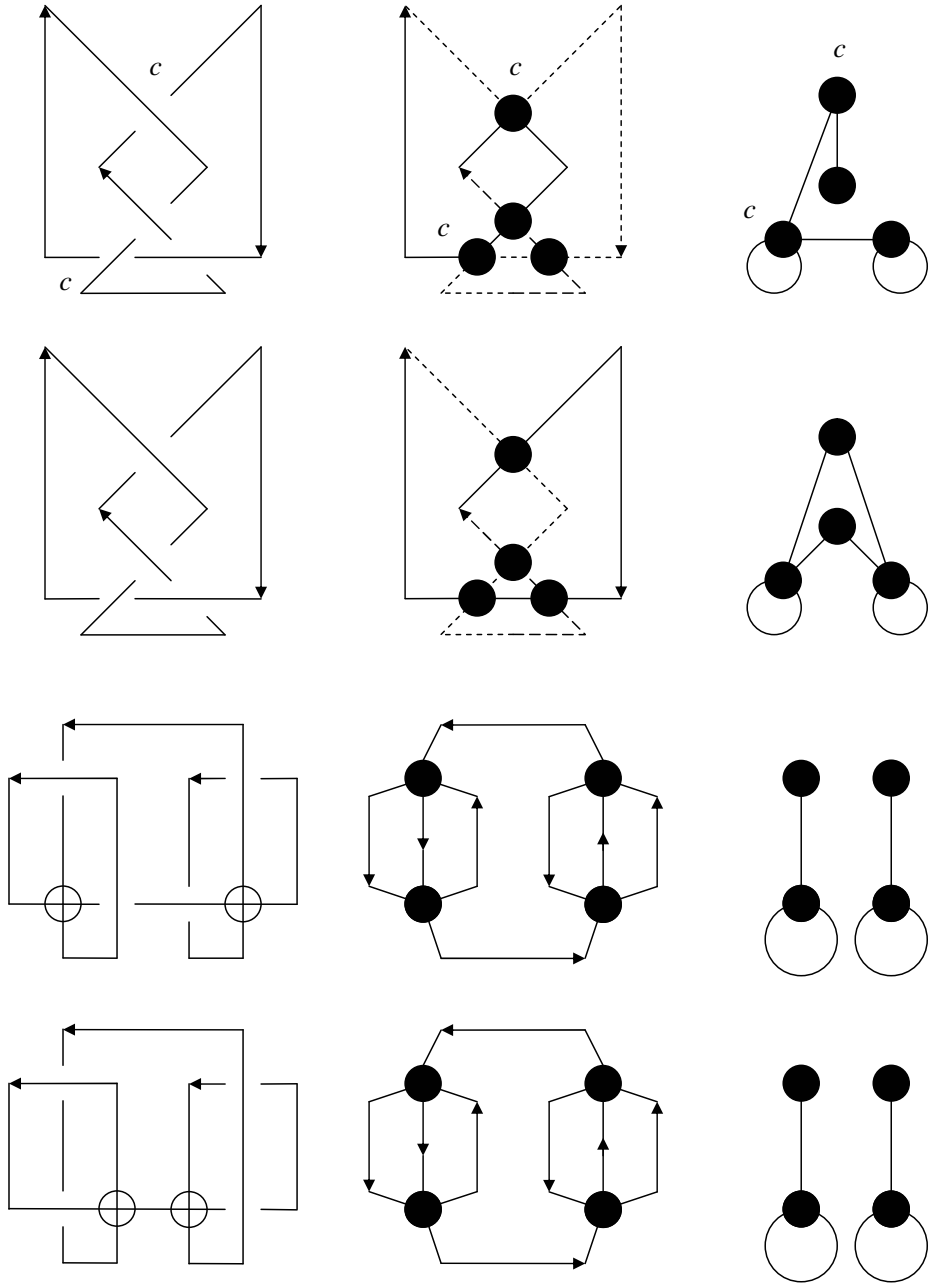}%
\caption{In the top two rows, different marked versions of the same diagram
give rise to different Euler circuits and looped interlacement graphs. In the
bottom two rows, different diagrams give rise to the same universe and looped
interlacement graph.}%
\label{lbracfk1}%
\end{center}
\end{figure}

\section{The extended Cohn-Lempel equality}

Cohn and Lempel \cite{CL} described the number of orbits in a finite set under
a certain kind of permutation using the nullity of an associated binary
matrix. When applied to a connected 2-in, 2-out digraph $\vec{U}$\ their
equality uses the nullity of a submatrix of an interlace matrix to calculate
the number of circuits that appear in a given partition of the edge-set of
$\vec{U}$ into directed circuits. This equality was implicitly rediscovered as
part of the combinatorial theory of the interlace polynomial (Theorem 24 of
\cite{A1} or \cite{A2}), and its usefulness in knot theory was pointed out by
Soboleva \cite{S}. Other equalities similar to that of Cohn and Lempel have
been used by knot theorists: Zulli \cite{Z} used a formula that always refers
to $n\times n$ matrices, and a more general version was stated by Mellor
\cite{Me} and Lando \cite{L}.

\bigskip

An extended version of the Cohn-Lempel equality that includes all of these and
applies to arbitrary circuit partitions in arbitrary 4-regular graphs is
proven in \cite{T}. Let $U$ be an undirected 4-regular graph with
$V(U)=\{v_{1},...,v_{n}\}$, and let $C$ be an Euler system for $U$. Choose an
orientation for each circuit in $C$, and let $\vec{U}$ be the 2-in, 2-out
digraph obtained from $U$ by directing all edges according to these
orientations. Suppose $P$ is a partition of $E(U)$ into undirected circuits;
$P$ should also contain every free loop $U$ might have. Suppose that we follow
a directed edge $e$ of $\vec{U}$ toward a vertex $v_{i}$. If the circuit of
$P$ that contains $e$ leaves $v_{i}$ along the edge $C$ uses to leave $v_{i}$
after arriving along $e$, we say $P$ \textit{follows }$C$ \textit{through}
$v_{i}$. If the circuit of $P$ that contains $e$ leaves $v_{i}$ along the
other edge directed away from $v_{i}$, we say $P$ is
\textit{orientation-consistent at} $v_{i}$\textit{\ but does not follow} $C$.
The last possibility is that the circuit of $P$ that contains $e$ leaves
$v_{i}$ along the other edge directed toward $v_{i}$; in this case we say $P$
is \textit{orientation-inconsistent at} $v_{i}$. (Note that the description of
$P$ at $v_{i}$ is unchanged if we begin with the other edge directed toward
$v_{i}$, and it is not affected by the choice of orientations for the circuits
of $C$ or $P$.) A matrix $I_{P}=I_{P}(U,C)$ is associated to $P$: If $P$
follows $C$ through $v_{i}$ then the row and column of $I(U,C)$ corresponding
to $v_{i}$ are removed; if $P$ is orientation-consistent at $v_{i}$ but does
not follow $C$ then the row and column of $I(U,C)$ corresponding to $v_{i}$
are retained; and if $P$ is orientation-inconsistent at $v_{i}$ then the
off-diagonal entries of the row and column of $I(U,C)$ corresponding to
$v_{i}$ are retained but their common diagonal entry is changed from $0$ to
$1$.

\begin{proposition}
(Extended Cohn-Lempel equality) Let $U$ be an undirected, 4-regular graph with
$c(U)$ connected components, and let $P$ be a partition of $E(U)$ into
undirected circuits; $P$ must also contain every free loop of $U$. Then the
$GF(2)$-nullity of $I_{P}$ is
\[
\nu(I_{P})=\left|  P\right|  -c(U).
\]

\end{proposition}

We do not give a proof here; the interested reader can find one in \cite{T}.
Here is an example, though. Consider the complete bipartite graph $K_{4,4}$,
with vertices denoted 1, 2, 3, 4, 5, 6, 7 and 8; vertex $i$ is adjacent to
vertex $j$ if and only if $i\not \equiv j$ (mod 2). Let $C$ contain the Euler
circuit 1234567814725836. If $P$ follows $C$ at vertices 1 and 4, is
orientation-inconsistent at vertices 2, 5 and 7, and is orientation-consistent
but does not follow $C$ at vertices 3, 6 and 8 then
\[
\nu(I_{P}(U,C))=\nu%
\begin{pmatrix}
1 & 1 & 1 & 1 & 0 & 1\\
1 & 0 & 0 & 1 & 0 & 0\\
1 & 0 & 1 & 1 & 0 & 1\\
1 & 1 & 1 & 0 & 0 & 0\\
0 & 0 & 0 & 0 & 1 & 1\\
1 & 0 & 1 & 0 & 1 & 0
\end{pmatrix}
=0,
\]
so $\left\vert P\right\vert =1$. The one circuit in $P$ is the Euler circuit
1278345236741856. The partition $P^{\prime}$ that disagrees with $P$ only by
following $C$ at 2 and 6 corresponds to the matrix $I_{P^{\prime}}(U,C)$
obtained by removing the first and fourth rows and columns of $I_{P}(U,C)$, so
$\nu(I_{P^{\prime}}(U,C))=2$. $P^{\prime}$ contains the circuits 1236, 147658
and 254387.

\section{A bracket polynomial for marked graphs}

\begin{definition}
\label{adjmatrix} If $G$ is a graph with $V(G)=\{v_{1},...,v_{n}\}$ then the
\emph{Boolean adjacency matrix} of $G$ is the $n\times n$ matrix
$\mathcal{A}(G)$ with the following entries in $GF(2)$: $\mathcal{A}%
(G)_{ii}=1$ if and only if $v_{i}$ is looped, and if $i\neq j$ then
$\mathcal{A}(G)_{ij}=1$ if and only if $v_{i}$ is adjacent to $v_{j}$.
\end{definition}

\begin{definition}
A \emph{marked graph} is a graph $G$ given with a partition of $V(G)$ into two subsets.
\end{definition}

In our discussion we call the vertices in one cell of the partition
\textit{unmarked}, and vertices in the other cell \textit{marked} (or
\textit{marked with the letter} $c$).

\begin{definition}
\label{adjT} Suppose $G$ is a marked graph with $V(G)=\{v_{1},...,v_{n}\}$ and
$T\subseteq V(G)$. Let $\Delta_{T}$ be the $n\times n$ diagonal matrix whose
$i^{th}$ diagonal entry is $1$ if and only if $v_{i}\in T$. Then
$\mathcal{A}(G)_{T}$ is the submatrix of $\mathcal{A}(G)+\Delta_{T}$ obtained
by removing the $i^{th}$ row and column if $v_{i}$ is marked and the $i^{th}$
diagonal entry of $\mathcal{A}(G)+\Delta_{T}$ is $0$.
\end{definition}

\begin{definition}
\label{bracket} The \emph{marked-graph bracket polynomial} of a marked graph
$G$ with $\phi$ free loops is
\[
\lbrack G]=d^{\phi}\cdot\sum_{T\subseteq V(G)}A^{n-\left\vert T\right\vert
}B^{\left\vert T\right\vert }d^{\nu(\mathcal{A}(G)_{T})},
\]
where $\nu$ denotes the nullity of matrices with entries in $GF(2)$.
\end{definition}

A simple consequence of Definitions \ref{adjT} and \ref{bracket} is that the
marked-graph bracket is multiplicative on disjoint unions.

\begin{proposition}
\label{product} If $G$ is the union of disjoint subgraphs $G_{1}$ and $G_{2}$
then $[G]=[G_{1}]\cdot\lbrack G_{2}].$
\end{proposition}

\begin{proof}
$G$ has $\phi=\phi_{1}+\phi_{2}$ free loops. Also, if $T\subseteq V(G)$ and
$T_{i}=T\cap V(G_{i})$ then
\[
\mathcal{A}(G)_{T}=
\begin{pmatrix}
\mathcal{A}(G)_{T_{1}} & \mathbf{0}\\
\mathbf{0} & \mathcal{A}(G)_{T_{2}}%
\end{pmatrix}
\]
so $\nu(\mathcal{A}(G)_{T})=\nu(\mathcal{A}(G)_{T_{1}})+\nu(\mathcal{A}%
(G)_{T_{2}})$.
\end{proof}

\bigskip

Another simple property of the marked-graph bracket is that toggling loops in
$G$ has the effect of reversing the roles of $A$ and $B\,$\ in $[G]$.

\begin{proposition}
\label{looptoggle} Let $G+I$ denote the marked graph obtained from $G$ by
toggling all loops in $G$ (i.e., $G+I$ has loops at precisely those vertices
where $G$ does not). Then $[G+I](A,B,d)=[G](B,A,d).$
\end{proposition}

\begin{proof}
If $T\subseteq V(G)$ then $\mathcal{A}(G)_{T}=\mathcal{A}(G+I)_{V(G)-T}$, so
the contribution of $T$ to $[G](B,A,d)$ is the same as the contribution of
$V(G)-T$ to $[G+I](A,B,d)$.
\end{proof}

\bigskip

Recall that the Kauffman bracket polynomial of a classical or virtual link
diagram $D$ is a sum indexed by the states of $D$:
\[
\lbrack D]=\sum_{S}A^{a(S)}B^{b(S)}d^{c(S)-1}.
\]
A Kauffman state $S$ corresponds in an obvious way to a partition $P(S)$ of
$E(U)$ into undirected circuits, with the proviso that every free loop of $U$
is included in both $S$ and $P(S)$. If $C$ is a directed Euler system for
$\vec{U}$ then we associate with each Kauffman state $S$ of $D$ the subset
$T(S)\subseteq V(\mathcal{L}(D,C))$ consisting of the vertices corresponding
to crossings where the state $S$ involves the $B$ smoothing. Observe that at
an unmarked positive crossing, the $B$ smoothing corresponds to the transition
that is orientation-inconsistent with $C$, and the $A$ smoothing corresponds
to the transition that is orientation-consistent with $C$ without following
$C$; at an unmarked negative crossing the $B$ smoothing is
orientation-consistent with $C$ without following it, and the $A$ smoothing is
orientation-inconsistent with $C$. At a positive crossing marked $c$ the $B$
smoothing is orientation-inconsistent with $C$ and the $A$ smoothing follows
$C$; at a negative crossing marked $c$ the $B$ smoothing follows $C$ and the
$A$ smoothing is orientation-inconsistent with $C$. Consequently if $S$ is a
Kauffman state of $D$ then $\mathcal{A}(\mathcal{L}(D,C))_{T(S)}$ is the same
as the matrix denoted $I_{P(S)}(U,C)$ in the preceding section. The extended
Cohn-Lempel equality then tells us that $\left|  P(S)\right|  =c(S)=\nu
(\mathcal{A}(\mathcal{L}(D,C))_{T(S)}+c(U).$ As $\mathcal{L}(D,C)$ has
$\phi=c(U)-1$ free loops, we conclude the following.

\begin{proposition}
\label{brackets} Suppose $D$ is a diagram of an oriented classical or virtual
link $L$, $U$ is the universe of $D$ and $C$ is a directed Euler system for
$\vec{U}$. Then the Kauffman bracket polynomial of $D$ and the marked-graph
bracket polynomial of $\mathcal{L}(D,C)$ are the same.
\end{proposition}

It follows that $[\mathcal{L}(D,C)]$ is independent of the choice of the
particular directed Euler system $C$. This independence actually follows from
Proposition \ref{ktransform}, Lemma \ref{transform} and a more general result
given in Theorem \ref{marked pivot}.

\begin{definition}
\label{pivot} Suppose $G$ is a marked graph and $v,w$ are adjacent in $G$. The
\emph{marked pivot} $G_{c}^{vw}$ is the marked graph obtained from $G^{vw}$ by
toggling the marks on $v$ and $w$ separately (i.e., removing a mark $c$ where
there is one, and inserting a mark $c$ where there is none), and interchanging
the neighbors of $v$ and $w$.
\end{definition}

Like an ordinary pivot, a marked pivot has no effect on loops or free loops.%

\begin{figure}
[ptb]
\begin{center}
\includegraphics[
trim=0.666277in 7.485691in 0.000000in 0.870727in,
height=1.785in,
width=5.7104in
]%
{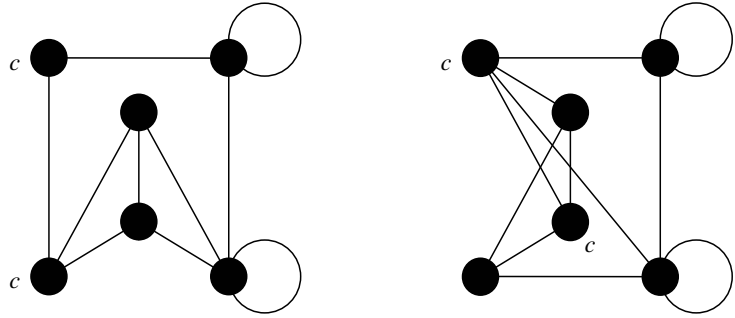}%
\caption{A marked pivot toggles the marks on two adjacent vertices, and
changes some adjacencies involving those vertices and their neighbors.}%
\label{lbracpiv}%
\end{center}
\end{figure}

\begin{theorem}
\label{marked pivot} Suppose $G$ is a marked graph and $v,w$ are adjacent in
$G$. Then $[G]=[G_{c}^{vw}].$
\end{theorem}

\begin{proof}
For convenience we presume $V(G)=\{v_{1},...,v_{n}\}$ with $v_{1}=v$ and
$v_{2}=w$. The proposition is proven by showing that the nullities of
$\mathcal{A}(G)_{T}$ and $\mathcal{A}(G_{c}^{vw})_{T}$ are the same for every
subset $T\subseteq V(G)$.

Suppose the first two diagonal entries of $\mathcal{A}(G)+\Delta_{T}$ are both
1. Then
\[
\mathcal{A}(G)_{T}=\mathcal{A}(G)+\Delta_{T}=
\begin{pmatrix}
1 & 1 & \mathbf{1} & \mathbf{1} & \mathbf{0} & \mathbf{0}\\
1 & 1 & \mathbf{1} & \mathbf{0} & \mathbf{1} & \mathbf{0}\\
\mathbf{1} & \mathbf{1} & M_{11} & M_{12} & M_{13} & M_{14}\\
\mathbf{1} & \mathbf{0} & M_{21} & M_{22} & M_{23} & M_{24}\\
\mathbf{0} & \mathbf{1} & M_{31} & M_{32} & M_{33} & M_{34}\\
\mathbf{0} & \mathbf{0} & M_{41} & M_{42} & M_{43} & M_{44}%
\end{pmatrix}
\]
for appropriate submatrices $M_{ij}$, where a bold numeral indicates a row or
column whose entries are all equal. Adding the first row of $\mathcal{A}%
(G)_{T}$ to every row in the third and fifth blocks of rows (those containing
$M_{11}$ and $M_{31}$), and then adding the second row to every row in the
third and fourth blocks of rows, we conclude that the nullity of
$\mathcal{A}(G)_{T}$ is the same as that of
\[%
\begin{pmatrix}
1 & 1 & \mathbf{1} & \mathbf{1} & \mathbf{0} & \mathbf{0}\\
1 & 1 & \mathbf{1} & \mathbf{0} & \mathbf{1} & \mathbf{0}\\
\mathbf{1} & \mathbf{1} & M_{11} & \bar{M}_{12} & \bar{M}_{13} & M_{14}\\
\mathbf{0} & \mathbf{1} & \bar{M}_{21} & M_{22} & \bar{M}_{23} & M_{24}\\
\mathbf{1} & \mathbf{0} & \bar{M}_{31} & \bar{M}_{32} & M_{33} & M_{34}\\
\mathbf{0} & \mathbf{0} & M_{41} & M_{42} & M_{43} & M_{44}%
\end{pmatrix}
,
\]
which is $\mathcal{A}(G_{c}^{vw})_{T}$ up to permutation of the rows and columns.

Suppose now that one of the first two diagonal entries of $\mathcal{A}%
(G)+\Delta_{T}$ is 1, and the other is 0. If $v$ and $w$ are both unmarked
then
\[
\mathcal{A}(G)_{T}=\mathcal{A}(G)+\Delta_{T}=
\begin{pmatrix}
0 & 1 & \mathbf{1} & \mathbf{1} & \mathbf{0} & \mathbf{0}\\
1 & 1 & \mathbf{1} & \mathbf{0} & \mathbf{1} & \mathbf{0}\\
\mathbf{1} & \mathbf{1} & M_{11} & M_{12} & M_{13} & M_{14}\\
\mathbf{1} & \mathbf{0} & M_{21} & M_{22} & M_{23} & M_{24}\\
\mathbf{0} & \mathbf{1} & M_{31} & M_{32} & M_{33} & M_{34}\\
\mathbf{0} & \mathbf{0} & M_{41} & M_{42} & M_{43} & M_{44}%
\end{pmatrix}
.
\]
Adding the first row of $\mathcal{A}(G)_{T}$ to every row in the third and
fifth blocks of rows, and then adding the second row to every row in the third
and fourth blocks, we conclude that the nullity of $\mathcal{A}(G)_{T}$ is the
same as that of
\[%
\begin{pmatrix}
0 & 1 & \mathbf{1} & \mathbf{1} & \mathbf{0} & \mathbf{0}\\
1 & 1 & \mathbf{1} & \mathbf{0} & \mathbf{1} & \mathbf{0}\\
\mathbf{0} & \mathbf{1} & M_{11} & \bar{M}_{12} & \bar{M}_{13} & M_{14}\\
\mathbf{0} & \mathbf{1} & \bar{M}_{21} & M_{22} & \bar{M}_{23} & M_{24}\\
\mathbf{0} & \mathbf{0} & \bar{M}_{31} & \bar{M}_{32} & M_{33} & M_{34}\\
\mathbf{0} & \mathbf{0} & M_{41} & M_{42} & M_{43} & M_{44}%
\end{pmatrix}
.
\]
Removing the first column and second row of this matrix does not change its
nullity, and yields $\mathcal{A}(G_{c}^{vw})_{T}$ up to permutation of rows
and columns. The argument can be reversed if $v$ and $w$ are both marked in
$G$, as they are then both unmarked in $G_{c}^{vw}$.

Also, if precisely one of $v,w$ is marked in $G$ then the marked one will
correspond to the diagonal entry 1 in one of $G,G_{c}^{vw}$ and the diagonal
entry 0 in the other of $G,G_{c}^{vw}$. The argument just given applies again,
possibly with the roles of $G$ and $G_{c}^{vw}$ reversed.

Suppose now that the first two diagonal entries of $\mathcal{A}(G)+\Delta_{T}$
are both 0, and suppose that precisely one of $v,w$ ($v$, say) is marked in
$G$. Then
\[
\mathcal{A}(G)+\Delta_{T}=
\begin{pmatrix}
0 & 1 & \mathbf{1} & \mathbf{1} & \mathbf{0} & \mathbf{0}\\
1 & 0 & \mathbf{1} & \mathbf{0} & \mathbf{1} & \mathbf{0}\\
\mathbf{1} & \mathbf{1} & M_{11} & M_{12} & M_{13} & M_{14}\\
\mathbf{1} & \mathbf{0} & M_{21} & M_{22} & M_{23} & M_{24}\\
\mathbf{0} & \mathbf{1} & M_{31} & M_{32} & M_{33} & M_{34}\\
\mathbf{0} & \mathbf{0} & M_{41} & M_{42} & M_{43} & M_{44}%
\end{pmatrix}
\]
and
\[
\mathcal{A}(G)_{T}=
\begin{pmatrix}
0 & \mathbf{1} & \mathbf{0} & \mathbf{1} & \mathbf{0}\\
\mathbf{1} & M_{11} & M_{12} & M_{13} & M_{14}\\
\mathbf{0} & M_{21} & M_{22} & M_{23} & M_{24}\\
\mathbf{1} & M_{31} & M_{32} & M_{33} & M_{34}\\
\mathbf{0} & M_{41} & M_{42} & M_{43} & M_{44}%
\end{pmatrix}
\]
for appropriate submatrices $M_{ij}$. Adding the first row of $\mathcal{A}%
(G)_{T}$ to every row in the second and third blocks of rows, and then adding
the first column to every column in the second and third blocks of columns, we
conclude that the nullity of $\mathcal{A}(G)_{T}$ is the same as that of
\[%
\begin{pmatrix}
0 & \mathbf{1} & \mathbf{0} & \mathbf{1} & \mathbf{0}\\
\mathbf{1} & M_{11} & \bar{M}_{12} & \bar{M}_{13} & M_{14}\\
\mathbf{0} & \bar{M}_{21} & M_{22} & \bar{M}_{23} & M_{24}\\
\mathbf{1} & \bar{M}_{31} & \bar{M}_{32} & M_{33} & M_{34}\\
\mathbf{0} & M_{41} & M_{42} & M_{43} & M_{44}%
\end{pmatrix}
,
\]
which is $\mathcal{A}(G_{c}^{vw})_{T}$ up to a permutation of the rows and columns.

Suppose now that the first two diagonal entries of $\mathcal{A}(G)+\Delta_{T}$
are both 0 and that $v$ and $w$ are both unmarked in $G$. Then
\[
\mathcal{A}(G)+\Delta_{T}=\mathcal{A}(G)_{T}=
\begin{pmatrix}
0 & 1 & \mathbf{1} & \mathbf{1} & \mathbf{0} & \mathbf{0}\\
1 & 0 & \mathbf{1} & \mathbf{0} & \mathbf{1} & \mathbf{0}\\
\mathbf{1} & \mathbf{1} & M_{11} & M_{12} & M_{13} & M_{14}\\
\mathbf{1} & \mathbf{0} & M_{21} & M_{22} & M_{23} & M_{24}\\
\mathbf{0} & \mathbf{1} & M_{31} & M_{32} & M_{33} & M_{34}\\
\mathbf{0} & \mathbf{0} & M_{41} & M_{42} & M_{43} & M_{44}%
\end{pmatrix}
.
\]
Adding the first row of $\mathcal{A}(G)_{T}$ to every row in the third and
fifth blocks of rows, and then adding the second row to every row in the third
and fourth blocks, we conclude that the nullity of $\mathcal{A}(G)_{T}$ is the
same as that of
\[%
\begin{pmatrix}
0 & 1 & \mathbf{1} & \mathbf{1} & \mathbf{0} & \mathbf{0}\\
1 & 0 & \mathbf{1} & \mathbf{0} & \mathbf{1} & \mathbf{0}\\
\mathbf{0} & \mathbf{0} & M_{11} & \bar{M}_{12} & \bar{M}_{13} & M_{14}\\
\mathbf{0} & \mathbf{0} & \bar{M}_{21} & M_{22} & \bar{M}_{23} & M_{24}\\
\mathbf{0} & \mathbf{0} & \bar{M}_{31} & \bar{M}_{32} & M_{33} & M_{34}\\
\mathbf{0} & \mathbf{0} & M_{41} & M_{42} & M_{43} & M_{44}%
\end{pmatrix}
.
\]
Removing the first two rows and columns of this matrix does not change the
nullity, and yields $\mathcal{A}(G_{c}^{vw})_{T}$, up to a permutation of the
rows and columns.

If the first two diagonal entries of $\mathcal{A}(G)+\Delta_{T}$ are both 0
and $v$ and $w$ are both marked the same argument applies, with the roles of
$G$ and $G_{c}^{vw}$ reversed.
\end{proof}

\bigskip

Theorem \ref{marked pivot} will be helpful in the balance of the paper because
it allows us to focus on a limited number of special cases. For instance, we
will see later that when applied to different configurations of marked
vertices, the knot-theoretic Reidemeister moves give rise to different
marked-graph Reidemeister moves. Theorem \ref{marked pivot} allows us to
ignore some of those configurations (e.g., ones involving adjacent marked vertices).

\section{A recursion for the marked-graph bracket}

In this section we present a recursive description of the marked-graph bracket
polynomial. The recursion involves formulas that are only valid when certain
vertices are without marked neighbors, so an arbitrary marked graph must be
\textquotedblleft prepared\textquotedblright\ for the recursion by using
marked pivots to eliminate adjacencies between marked vertices.

\bigskip%

\begin{figure}
[tbh]
\begin{center}
\includegraphics[
trim=1.071155in 5.216878in 0.000000in 0.805476in,
height=3.531in,
width=5.4042in
]%
{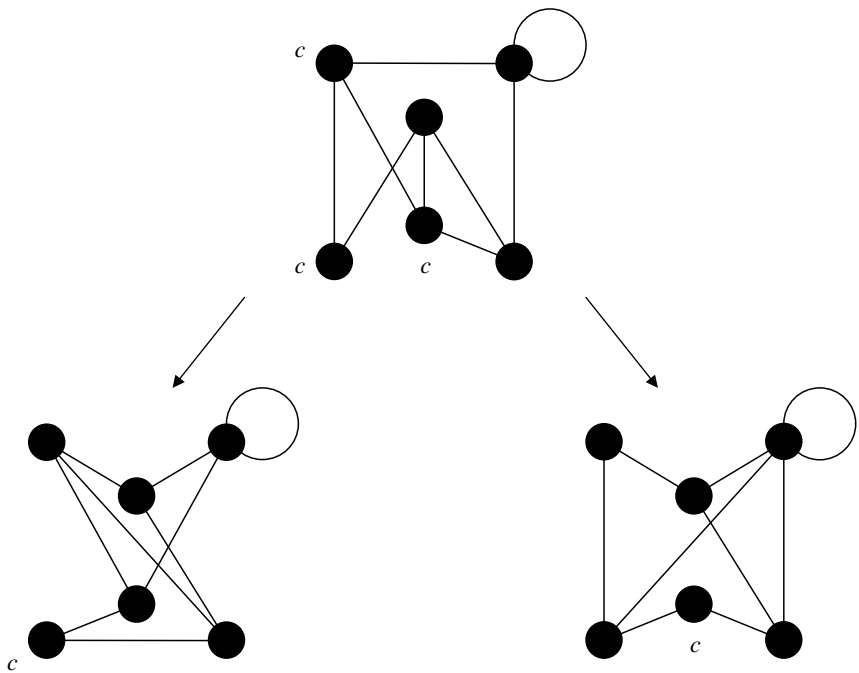}%
\caption{Marked pivots can be used to remove marks on adjacent marked
vertices. Different marked graphs may result, but they will all have the same
bracket polynomial.}%
\label{lbracl14b}%
\end{center}
\end{figure}

\begin{definition}
If $v$ is a vertex of a (marked) graph $G$ then the \emph{local complement}
$G^{v}$ is obtained from $G$ by toggling every edge incident only on
neighbor(s) of $v$.
\end{definition}

Recall that $v$ itself is not included among its neighbors, so edges incident
on $v$ in $G$ are preserved in $G^{v}$. We should also point out that there is
a different use of the term ``local complement,'' which is restricted to
simple graphs and does not involve loops \cite{Bc, B}; we follow \cite{A1, A2,
A} instead, and toggle both loops and non-loop edges incident on neighbors of
$v$ when constructing $G^{v}$.%

\begin{figure}
[ptb]
\begin{center}
\includegraphics[
trim=0.668751in 7.024655in 0.000000in 1.342461in,
height=1.7746in,
width=5.7095in
]%
{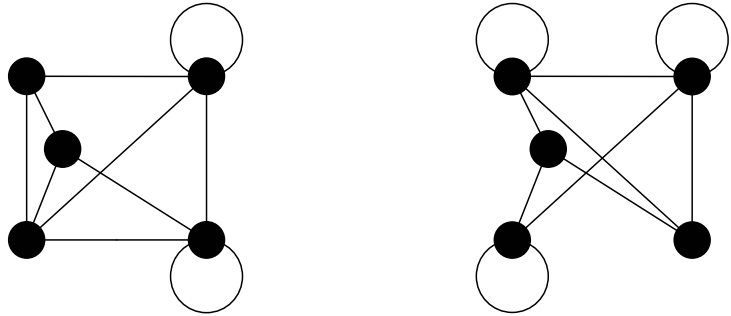}%
\caption{Local complementation at a vertex toggles loops and non-loop edges
that involve only neighbor(s) of that vertex.}%
\label{lbraclc}%
\end{center}
\end{figure}

\begin{theorem}
\label{recursion} The marked-graph bracket polynomial of a marked graph $G$
can be calculated recursively using the following steps.

(a) If $v$ and $w$ are marked neighbors, replace $G$ with $G_{c}^{vw}$.

(b) Suppose $v$ is unlooped and marked, and no neighbor of $v$ is marked.
Then
\[
\lbrack G]=A[G-v]+B[G^{v}-v],
\]
where $G-v$ is obtained from $G$ by removing $v$ and every edge incident on
$v$.

(c) Suppose $v$ is looped and marked, and no neighbor of $v$ is marked. Then
\[
\lbrack G]=B[G-v]+A[G^{v}-v].
\]

(d) If $v$ is looped and no neighbor of $v$ is marked, then
\[
\lbrack G]=A^{-1}B[G-\{v,v\}]+(A-A^{-1}B^{2})[G^{v}-v].
\]
Here $G-\{v,v\}$ is obtained from $G$ by removing the loop at $v$.

(e) Let $v$ and $w$ be adjacent, unlooped, unmarked vertices. If no neighbor
of $v$ is marked then
\[
\lbrack G]=A^{2}[G^{vw}-v-w]+AB[(G^{vw})^{v}-v-w]+B[G^{v}-v].
\]

(f) Suppose $G$ is a graph whose edges are all loops. If $G$ has $\phi$ free
loops, $m_{1}$ unlooped marked vertices, $m_{2}$ looped marked vertices,
$n_{1}$ unlooped unmarked vertices and $n_{2}$ looped unmarked vertices then
\[
\lbrack G]=d^{\phi}(A+B)^{m_{1}+m_{2}}(Ad+B)^{n_{1}}(A+Bd)^{n_{2}}\text{.}%
\]

\end{theorem}

\begin{proof}
Theorem \ref{marked pivot} tells us that an application of part (a) does not
change the marked-graph bracket polynomial, and the formula of part (f)
follows directly from Definition \ref{bracket}. As free loops are not affected
by local complementation, pivots or marked pivots, the corresponding powers of
$d$ can simply be factored out of the formulas of parts (b)-(e); hence it
suffices to verify (b)-(e) for graphs without free loops.

In part (b) $v$ is a marked, unlooped vertex with no marked neighbor. Let
\[
S=\sum_{v\in T\subseteq V(G)}A^{n-\left|  T\right|  }B^{\left|  T\right|
}d^{\nu(\mathcal{A}(G)_{T})}.
\]
If $v\in T\subseteq V(G)$ then using elementary row operations, we see that
\begin{align*}
\nu(\mathcal{A}(G)_{T})  & =\nu\left(
\begin{array}
[c]{ccc}%
1 & \mathbf{1} & \mathbf{0}\\
\mathbf{1} & M_{11} & M_{12}\\
\mathbf{0} & M_{21} & M_{22}%
\end{array}
\right) \\
& =\nu\left(
\begin{array}
[c]{ccc}%
1 & \mathbf{1} & \mathbf{0}\\
\mathbf{0} & \bar{M}_{11} & M_{12}\\
\mathbf{0} & M_{21} & M_{22}%
\end{array}
\right)  =\nu\left(
\begin{array}
[c]{cc}%
\bar{M}_{11} & M_{12}\\
M_{21} & M_{22}%
\end{array}
\right)
\end{align*}
for appropriate matrices $M_{ij}$. As $v$ has no marked neighbor, the toggling
of diagonal entries in $\bar{M}_{11}$ is compatible with Definition
\ref{adjT}, and we conclude that
\[
\nu(\mathcal{A}(G)_{T})=\nu(\mathcal{A}(G^{v}-v)_{T-\{v\}}).
\]
This holds whenever $v\in T$, so $S=B[G^{v}-v]$. As $v$ is marked, Definition
\ref{adjT} has $\mathcal{A}(G)_{T}=\mathcal{A}(G-v)_{T}$ for every $T\subseteq
V(G)$ with $v\not \in T$. It follows that $[G]-S=A[G-v]$.

Part (c) follows from part (b) and Proposition \ref{looptoggle}.

We proceed to consider (d).\ Suppose $G$ has a looped vertex $v$ whose
neighbors are all unmarked. If $v\in T\subseteq V(G)$ then $\mathcal{A}%
(G)_{T}=\mathcal{A}(G-\{v,v\})_{T-\{v\}}$ and $\mathcal{A}(G)_{T-\{v\}}%
=\mathcal{A}(G-\{v,v\})_{T}$; these equalities hold whether or not $v$ is
marked. Let
\[
S=\sum_{v\in T\subseteq V(G)}A^{n-\left|  T\right|  }B^{\left|  T\right|
}d^{\nu(\mathcal{A}(G-\{v,v\})_{T})}.
\]
Then $[G]=AB^{-1}S+A^{-1}B\left(  [G-\{v,v\}]-S\right)  $, so $[G]=A^{-1}%
B[G-\{v,v\}]+(AB^{-1}-A^{-1}B)S$. The nullity argument used in the proof of
(b) shows that $S=B[(G-\{v,v\})^{v}-v]=B[G^{v}-v]$.

Turning to (e), suppose that $v,w\in V(G)$ are adjacent, unlooped, unmarked
vertices, and no neighbor of $v$ is marked. Let
\begin{align*}
S_{1}  & =\sum_{\substack{T\subseteq V(G) \\v\notin T,w\notin T}}A^{n-\left|
T\right|  }B^{\left|  T\right|  }d^{\nu(\mathcal{A}(G)_{T})}\text{ }\\
\text{and }S_{2}  & =\sum_{\substack{T\subseteq V(G) \\v\notin T,w\in
T}}A^{n-\left|  T\right|  }B^{\left|  T\right|  }d^{\nu(\mathcal{A}(G)_{T}%
)}\text{.}%
\end{align*}
As $v$ has no marked neighbor the nullity argument given for part (b) applies;
it tells us that $[G]-S_{1}-S_{2}=B[G^{v}-v]$.

Consider a subset $T\subseteq V(G)$ with $v,w\notin T$. Adding the first row
to those in the third and fifth blocks of rows and adding the second row to
those in the third and fourth blocks of rows, we see that
\begin{align*}
\nu(\mathcal{A}(G)_{T})  & =\nu\left(
\begin{array}
[c]{cccccc}%
0 & 1 & \mathbf{1} & \mathbf{1} & \mathbf{0} & \mathbf{0}\\
1 & 0 & \mathbf{1} & \mathbf{0} & \mathbf{1} & \mathbf{0}\\
\mathbf{1} & \mathbf{1} & M_{11} & M_{12} & M_{13} & M_{14}\\
\mathbf{1} & \mathbf{0} & M_{21} & M_{22} & M_{23} & M_{24}\\
\mathbf{0} & \mathbf{1} & M_{31} & M_{32} & M_{33} & M_{34}\\
\mathbf{0} & \mathbf{0} & M_{41} & M_{42} & M_{43} & M_{44}%
\end{array}
\right) \\
& ~\\
& =\nu\left(
\begin{array}
[c]{cccccc}%
0 & 1 & \mathbf{1} & \mathbf{1} & \mathbf{0} & \mathbf{0}\\
1 & 0 & \mathbf{1} & \mathbf{0} & \mathbf{1} & \mathbf{0}\\
\mathbf{0} & \mathbf{0} & M_{11} & \bar{M}_{12} & \bar{M}_{13} & M_{14}\\
\mathbf{0} & \mathbf{0} & \bar{M}_{21} & M_{22} & \bar{M}_{23} & M_{24}\\
\mathbf{0} & \mathbf{0} & \bar{M}_{31} & \bar{M}_{32} & M_{33} & M_{34}\\
\mathbf{0} & \mathbf{0} & M_{41} & M_{42} & M_{43} & M_{44}%
\end{array}
\right) \\
& ~\\
& =\nu\left(
\begin{array}
[c]{cccc}%
M_{11} & \bar{M}_{12} & \bar{M}_{13} & M_{14}\\
\bar{M}_{21} & M_{22} & \bar{M}_{23} & M_{24}\\
\bar{M}_{31} & \bar{M}_{32} & M_{33} & M_{34}\\
M_{41} & M_{42} & M_{43} & M_{44}%
\end{array}
\right)  =\nu(\mathcal{A}(G^{vw}-v-w)_{T}).
\end{align*}
It follows that $S_{1}=A^{2}[G^{vw}-v-w]$.

Now, consider a subset $T\subseteq V(G)$ with $v\notin T$ and $w\in T$. Adding
the first row to those in the fourth and fifth blocks of rows, and the second
row to those in the third and fourth blocks, we see that
\begin{align*}
\nu(\mathcal{A}(G)_{T})  & =\nu\left(
\begin{array}
[c]{cccccc}%
0 & 1 & \mathbf{1} & \mathbf{1} & \mathbf{0} & \mathbf{0}\\
1 & 1 & \mathbf{1} & \mathbf{0} & \mathbf{1} & \mathbf{0}\\
\mathbf{1} & \mathbf{1} & M_{11} & M_{12} & M_{13} & M_{14}\\
\mathbf{1} & \mathbf{0} & M_{21} & M_{22} & M_{23} & M_{24}\\
\mathbf{0} & \mathbf{1} & M_{31} & M_{32} & M_{33} & M_{34}\\
\mathbf{0} & \mathbf{0} & M_{41} & M_{42} & M_{43} & M_{44}%
\end{array}
\right) \\
& ~\\
& =\nu\left(
\begin{array}
[c]{cccccc}%
0 & 1 & \mathbf{1} & \mathbf{1} & \mathbf{0} & \mathbf{0}\\
1 & 0 & \mathbf{1} & \mathbf{0} & \mathbf{1} & \mathbf{0}\\
\mathbf{0} & \mathbf{0} & \bar{M}_{11} & M_{12} & \bar{M}_{13} & M_{14}\\
\mathbf{0} & \mathbf{0} & M_{21} & \bar{M}_{22} & \bar{M}_{23} & M_{24}\\
\mathbf{0} & \mathbf{0} & \bar{M}_{31} & \bar{M}_{32} & M_{33} & M_{34}\\
\mathbf{0} & \mathbf{0} & M_{41} & M_{42} & M_{43} & M_{44}%
\end{array}
\right) \\
& ~\\
& =\nu\left(
\begin{array}
[c]{cccc}%
\bar{M}_{11} & M_{12} & \bar{M}_{13} & M_{14}\\
M_{21} & \bar{M}_{22} & \bar{M}_{23} & M_{24}\\
\bar{M}_{31} & \bar{M}_{32} & M_{33} & M_{34}\\
M_{41} & M_{42} & M_{43} & M_{44}%
\end{array}
\right)  =\nu(\mathcal{A}((G^{vw})^{v}-v-w)_{T-\{w\}}).
\end{align*}
The last equality is justified because $v$ has no marked neighbors in $G$, so
the toggling of diagonal entries in $M_{11}$ and $M_{22}$ is consistent with
Definition \ref{adjT}. It follows that $S_{2}=AB[(G^{vw})^{v}-v-w].$
\end{proof}

\bigskip

One way to implement Theorem \ref{recursion} is to follow this simple outline:
first use (a) to remove marks on adjacent vertices; then use (b) and (c) to
remove marked vertices; then use (d) to remove loops; and finally use (e) to
remove the remaining edges. Typically, this will not result in the most
efficient recursive calculation of the bracket polynomial of a given marked
graph. For one thing, an implementation of \textit{any} branching algorithm
can be simplified if one combines terms that correspond to isomorphic
structures as they arise during the recursion. Also, in our case it is
obviously convenient to use part (f) when $G$ has only isolated vertices,
rather than first using parts (c) and (d) to deal with those vertices that
happen to be looped. More generally, Proposition \ref{looptoggle} will
sometimes yield a more convenient graph; for instance, if the vertices of $G$
are all looped it is certainly simpler to unloop all the vertices at once
rather than using parts (c) and (d) of Theorem \ref{recursion} repeatedly. In
addition, an efficient computation would utilize Proposition \ref{product}.
The computational cost of an instance of part (b), (c), (d) or (e) varies
according to the structure of the resulting graphs $G^{v}$, $G^{vw}$ and
$(G^{vw})^{v}$; there are generally many choices of vertices at which parts
(b)-(e) might be applied, and it is not easy to guess which of these choices
might be most convenient. Finally, there are other identities which might be
used in particular situations (when two unmarked looped vertices are adjacent,
when an unmarked looped vertex is adjacent to a marked unlooped vertex, etc.)
and might sometimes appear in more efficient computations than an
implementation of Theorem \ref{recursion}. Two of these other identities are
given in Corollary \ref{reverse} and Proposition \ref{old recursion}.

\section{Applying Theorem \ref{recursion} to links}%

\begin{figure}
[ptbh]
\begin{center}
\includegraphics[
trim=0.791866in 4.094698in 0.000000in 0.135884in,
height=4.8801in,
width=5.6204in
]%
{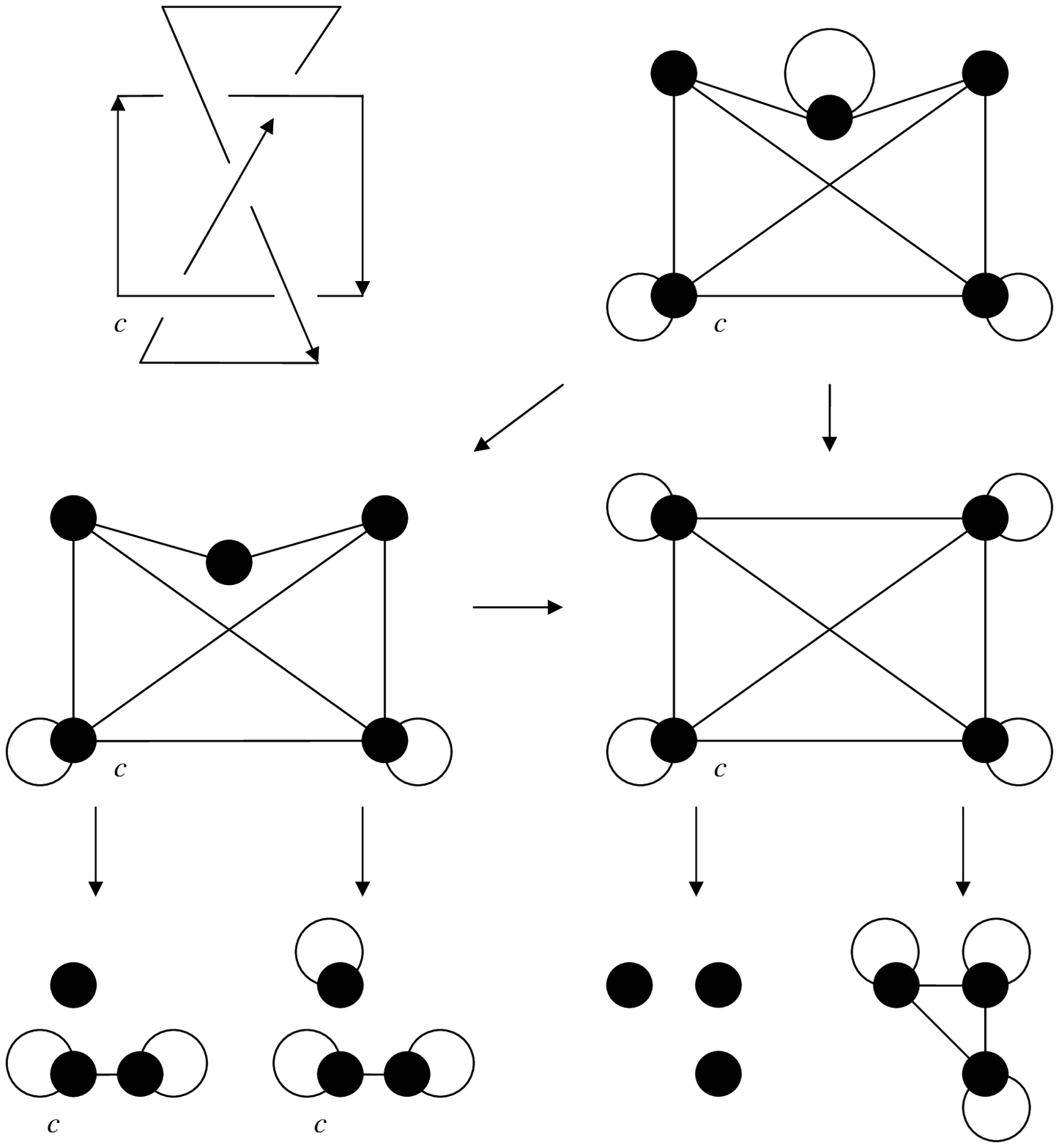}%
\caption{A bracket calculation.}%
\label{lbracfw}%
\end{center}
\end{figure}

Figure \ref{lbracfw} illustrates a calculation of the Kauffman bracket of a
diagram of Whitehead's link using Theorem \ref{recursion}. The marked link
diagram $D$ appears at the top left of the figure, with $\mathcal{L}(D,C)$ at
the top right. The two arrows directed downward from $\mathcal{L}(D,C)$ point
to the two graphs that result when part (d) of Theorem \ref{recursion} is
applied, with $v$ the only vertex of $\mathcal{L}(D,C)$ that is not adjacent
to the marked vertex. Part (e) of Theorem \ref{recursion} is then applied to
the left-hand graph in the middle row of the figure, with $v$ again the only
vertex of $\mathcal{L}(D,C)$ that is not adjacent to the marked vertex. As
indicated in the figure, the second term of the first step and the third term
of the second step refer to the same graph $G^{v}-v$. The bracket polynomial
of this graph is calculated using part (c): $[G^{v}-v]=A[(G^{v}-v)^{y}%
-y]+B[G^{v}-v-y]$, where $y$ is the marked vertex. The result of the
calculation is
\begin{gather*}
\lbrack D]=AB(Ad+B)(A^{2}d+2AB+B^{2}d)+B^{2}(A+Bd)(A^{2}d+2AB+B^{2}d)\\
+A(A(Ad+B)^{3}+B(A^{3}d^{2}+3A^{2}Bd+3AB^{2}+B^{3}d))\\
=d^{3}A^{5}+5A^{4}Bd^{2}+10A^{3}B^{2}d+8A^{2}B^{3}+2A^{2}B^{3}d^{2}%
+5AB^{4}d+B^{5}d^{2}.
\end{gather*}

Another calculation of the same bracket polynomial is indicated in\ Figure
\ref{lbracfw1}. This time the first step is an application of part (c),
$[G]=B[G-v]+A[G^{v}-v]$ with $v$ the marked vertex. The second step is an
application of part (d), and the third step is an application of part (e). The
result of the calculation is
\begin{gather*}
\lbrack D]=B[G-v]+A(Ad+B)(A^{3}d^{2}+3A^{2}Bd+3AB^{2}+B^{3}d)\\
=BA^{-1}B\cdot(A^{2}(Ad+B)(A+Bd)+AB(Ad+B)^{2})\\
+BA^{-1}B\cdot B(A^{3}+3A^{2}Bd+AB^{2}d^{2}+2AB^{2}+B^{3}d)\\
+B(A-A^{-1}B^{2})(A^{3}d^{2}+3A^{2}Bd+3AB^{2}+B^{3}d)\\
+A(Ad+B)(A^{3}d^{2}+3A^{2}Bd+3AB^{2}+B^{3}d).
\end{gather*}
%

\begin{figure}
[ptb]
\begin{center}
\includegraphics[
trim=1.337923in 2.943432in 1.072318in 0.000000in,
height=5.8427in,
width=4.4079in
]%
{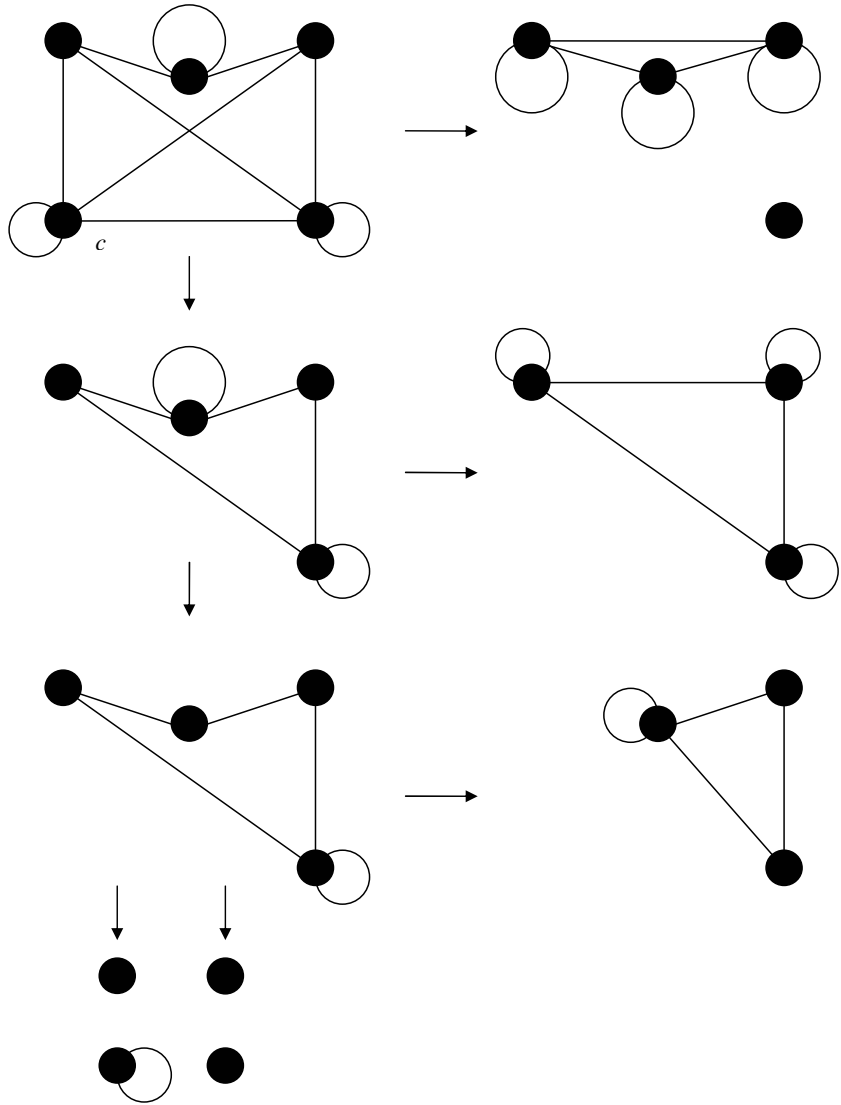}%
\caption{Another bracket calculation.}%
\label{lbracfw1}%
\end{center}
\end{figure}

The recursive calculation of the Kauffman bracket $[D]=$ $[\mathcal{L}(D,C)]$
provided by Theorem \ref{recursion} has not appeared before, but the
individual steps are all related to familiar properties of the Kauffman
bracket or Jones polynomial. The basic recursion of Kauffman's bracket,
$[D]=A[D_{A}]+B[D_{B}]$, is extended to marked graphs in parts (b) and (c) of
Theorem \ref{recursion}. As discussed in Section 4 of \cite{TZ}, the formula
of part (d) of Theorem \ref{recursion} corresponds to the Jones polynomial's
braid-plat formula (denoted $tV_{-1}-V_{1}=t^{3q}(t-1)V_{\infty}$ in
\cite{BK}) and the Kauffman bracket's switching formula (denoted $A\chi
-A^{-1}\bar{\chi}=(A^{2}-A^{-2})\asymp$ in \cite{Kd}), and the formula of part
(e) corresponds to a double use of the basic recursion of the Kauffman
bracket, $[D]=A^{2}[D_{AA}]+AB[D_{AB}]+B[D_{B}]$.

\bigskip

In addition, Theorem \ref{recursion} implies the reversing property of the
Jones polynomial \cite{LM, M}. The proof is simple: consider a link diagram in
which a certain component is incident on only one marked crossing, and apply
Corollary \ref{reverse} to conclude that the Kauffman bracket is not changed
when the orientation of that component is reversed.

\begin{corollary}
\label{reverse}Suppose $v$ is looped and marked in $G$, and no neighbor of $v
$ is marked. Then $[G]=[G^{v}-\{v,v\}]$.
\end{corollary}

\begin{proof}
This follows immediately from parts (b) and (c) of Theorem \ref{recursion}.
\end{proof}

\bigskip

The recursion of Theorem \ref{recursion} is fundamentally different from the
original recursive description of the classical Jones polynomial \cite{Jo},
which involves reducing the unknotting number (a quantity that is not
generally well defined for virtual links). The basic formula $t^{-1}V_{L^{+}%
}-tV_{L^{-}}=(t^{1/2}-t^{-1/2})V_{L}$ of \cite{Jo} corresponds to one of the
\textquotedblleft other identities\textquotedblright\ mentioned in the
preceding section:

\begin{proposition}
\label{old recursion} If $v\in V(G)$ is looped and marked then
\[
\lbrack G]=AB^{-1}[G-\{v,v\}]+(B-A^{2}B^{-1})[G-v].
\]

\end{proposition}

\begin{proof}
If $v\in T\subseteq V(G)$ then $\mathcal{A}(G)_{T}=\mathcal{A}%
(G-\{v,v\})_{T-\{v\}}$ and $\mathcal{A}(G)_{T-\{v\}}=\mathcal{A}%
(G-\{v,v\})_{T}$. Consequently, if
\[
S=\sum_{v\in T\subseteq V(G)}A^{n-\left|  T\right|  }B^{\left|  T\right|
}d^{\nu(\mathcal{A}(G)_{T})}%
\]
then $[G-\{v,v\}]=AB^{-1}S+A^{-1}B([G]-S)$, so $[G]=AB^{-1}[G-\{v,v\}]-(A^{2}%
B^{-2}-1)S$. As $v$ is marked, Definition \ref{bracket} tells us that
$S=B[G-v]$.
\end{proof}

\bigskip

Observe that unlike the formulas given in Theorem \ref{recursion}, Proposition
\ref{old recursion} allows $v$ to have marked neighbors.

\section{$\Omega.2$ moves and the reduced bracket}

\begin{definition}
If $G$ is a marked graph then the \emph{reduced marked-graph bracket}
$\left\langle G\right\rangle $ is obtained from $[G]$ by the evaluations
$B\mapsto A^{-1}$ and $d\mapsto-A^{2}-A^{-2}$.
\end{definition}

\begin{proposition}
\label{2movea} Suppose $G$ is a marked graph with two unmarked vertices $v$
and $w$ such that $v$ is looped, $w$ is not, and $v$ and $w$ are adjacent to
the same vertices outside $\{v,w\}$. Then $\left\langle G\right\rangle
=\left\langle G-v-w\right\rangle $.
\end{proposition}

\begin{proof}
Suppose for convenience that $V(G)=\{v_{1},...,v_{n}\}$ with $v_{1}=v$ and
$v_{2}=w$. If $T\subseteq V(G)-\{v,w\}$ then let $T_{1}=T\cup\{v\}$,
$T_{2}=T\cup\{w\}$ and $T_{12}=T\cup\{v,w\}$.

If $v$ and $w$ are not adjacent then using elementary row and column
operations we see that
\begin{align*}
\nu(\mathcal{A}(G)_{T})  & =\nu\left(
\begin{array}
[c]{cccc}%
1 & 0 & \mathbf{1} & \mathbf{0}\\
0 & 0 & \mathbf{1} & \mathbf{0}\\
\mathbf{1} & \mathbf{1} & M_{11} & M_{12}\\
\mathbf{0} & \mathbf{0} & M_{21} & M_{22}%
\end{array}
\right)  =\nu\left(
\begin{array}
[c]{ccc}%
0 & \mathbf{1} & \mathbf{0}\\
\mathbf{1} & M_{11} & M_{12}\\
\mathbf{0} & M_{21} & M_{22}%
\end{array}
\right)  ,\\
\nu(\mathcal{A}(G)_{T_{1}})  & =\nu\left(
\begin{array}
[c]{cccc}%
0 & 0 & \mathbf{1} & \mathbf{0}\\
0 & 0 & \mathbf{1} & \mathbf{0}\\
\mathbf{1} & \mathbf{1} & M_{11} & M_{12}\\
\mathbf{0} & \mathbf{0} & M_{21} & M_{22}%
\end{array}
\right)  =1+\nu\left(
\begin{array}
[c]{ccc}%
0 & \mathbf{1} & \mathbf{0}\\
\mathbf{1} & M_{11} & M_{12}\\
\mathbf{0} & M_{21} & M_{22}%
\end{array}
\right)  ,\\
\nu(\mathcal{A}(G)_{T_{2}})  & =\nu\left(
\begin{array}
[c]{cccc}%
1 & 0 & \mathbf{1} & \mathbf{0}\\
0 & 1 & \mathbf{1} & \mathbf{0}\\
\mathbf{1} & \mathbf{1} & M_{11} & M_{12}\\
\mathbf{0} & \mathbf{0} & M_{21} & M_{22}%
\end{array}
\right)  =\nu\left(
\begin{array}
[c]{cc}%
M_{11} & M_{12}\\
M_{21} & M_{22}%
\end{array}
\right)  \text{ \ and}\\
\nu(\mathcal{A}(G)_{T_{12}})  & =\nu\left(
\begin{array}
[c]{cccc}%
0 & 0 & \mathbf{1} & \mathbf{0}\\
0 & 1 & \mathbf{1} & \mathbf{0}\\
\mathbf{1} & \mathbf{1} & M_{11} & M_{12}\\
\mathbf{0} & \mathbf{0} & M_{21} & M_{22}%
\end{array}
\right)  =\nu\left(
\begin{array}
[c]{ccc}%
0 & \mathbf{1} & \mathbf{0}\\
\mathbf{1} & M_{11} & M_{12}\\
\mathbf{0} & M_{21} & M_{22}%
\end{array}
\right)  .
\end{align*}
It follows that the image of $A^{2}d^{\nu(\mathcal{A}(G)_{T})}+ABd^{\nu
(\mathcal{A}(G)_{T_{1}})}+B^{2}d^{\nu(\mathcal{A}(G)_{T_{12}})}$ under the
evaluations $B\mapsto A^{-1}$ and $d\mapsto-A^{2}-A^{-2}$ is 0. The image of
the remaining term is the contribution of $T$ to the reduced bracket of
$G-v-w$, so $\left\langle G\right\rangle =\left\langle G-v-w\right\rangle $.

If $v$ and $w$ are adjacent then
\begin{align*}
\nu(\mathcal{A}(G)_{T})  & =\nu\left(
\begin{array}
[c]{cccc}%
1 & 1 & \mathbf{1} & \mathbf{0}\\
1 & 0 & \mathbf{1} & \mathbf{0}\\
\mathbf{1} & \mathbf{1} & M_{11} & M_{12}\\
\mathbf{0} & \mathbf{0} & M_{21} & M_{22}%
\end{array}
\right)  =\nu\left(
\begin{array}
[c]{ccc}%
1 & \mathbf{1} & \mathbf{0}\\
\mathbf{1} & M_{11} & M_{12}\\
\mathbf{0} & M_{21} & M_{22}%
\end{array}
\right)  ,\\
\nu(\mathcal{A}(G)_{T_{1}})  & =\nu\left(
\begin{array}
[c]{cccc}%
0 & 1 & \mathbf{1} & \mathbf{0}\\
1 & 0 & \mathbf{1} & \mathbf{0}\\
\mathbf{1} & \mathbf{1} & M_{11} & M_{12}\\
\mathbf{0} & \mathbf{0} & M_{21} & M_{22}%
\end{array}
\right)  =\nu\left(
\begin{array}
[c]{cc}%
M_{11} & M_{12}\\
M_{21} & M_{22}%
\end{array}
\right)  ,\\
\nu(\mathcal{A}(G)_{T_{2}})  & =\nu\left(
\begin{array}
[c]{cccc}%
1 & 1 & \mathbf{1} & \mathbf{0}\\
1 & 1 & \mathbf{1} & \mathbf{0}\\
\mathbf{1} & \mathbf{1} & M_{11} & M_{12}\\
\mathbf{0} & \mathbf{0} & M_{21} & M_{22}%
\end{array}
\right)  =1+\nu\left(
\begin{array}
[c]{ccc}%
1 & \mathbf{1} & \mathbf{0}\\
\mathbf{1} & M_{11} & M_{12}\\
\mathbf{0} & M_{21} & M_{22}%
\end{array}
\right)  ,\text{ \ and}\\
\nu(\mathcal{A}(G)_{T_{12}})  & =\nu\left(
\begin{array}
[c]{cccc}%
0 & 1 & \mathbf{1} & \mathbf{0}\\
1 & 1 & \mathbf{1} & \mathbf{0}\\
\mathbf{1} & \mathbf{1} & M_{11} & M_{12}\\
\mathbf{0} & \mathbf{0} & M_{21} & M_{22}%
\end{array}
\right)  =\nu\left(
\begin{array}
[c]{ccc}%
1 & \mathbf{1} & \mathbf{0}\\
\mathbf{1} & M_{11} & M_{12}\\
\mathbf{0} & M_{21} & M_{22}%
\end{array}
\right)  .
\end{align*}
The image of $A^{2}d^{\nu(\mathcal{A}(G)_{T})}+ABd^{\nu(\mathcal{A}(G)_{T_{2}%
})}+B^{2}d^{\nu(\mathcal{A}(G)_{T_{12}})}$ under the evaluations $B\mapsto
A^{-1}$ and $d\mapsto-A^{2}-A^{-2}$ is 0, and the image of the remaining term
is the contribution of $T$ to the reduced bracket of $G-v-w$.
\end{proof}

\begin{proposition}
\label{2moveb}Suppose $G$ is a marked graph with two adjacent vertices $v$ and
$w$ such that $v$ is looped and marked, $w$ is neither looped nor marked, and
$v$ is the only neighbor of $w$. If $v$ has no neighbor other than $w$ then
$\left\langle G\right\rangle =\left\langle G^{+}-v-w\right\rangle $, where
$G^{+}$ is obtained from $G$ by adjoining a free loop. Otherwise, let $z$ be a
neighbor of $v$ other than $w$. Then $\left\langle G\right\rangle
=\left\langle G_{c}^{vz}-v-w\right\rangle $.
\end{proposition}

\begin{proof}
If $v$ has no neighbor other than $w$ then Proposition \ref{product} tells us
that $\left\langle G\right\rangle =\left\langle G[\{v,w\}]\right\rangle
\cdot\left\langle G-v-w\right\rangle $, where $G[\{v,w\}]$ is the subgraph of
$G$ induced by $\{v,w\}$. A direct calculation shows that $[G[\{v,w\}]]=A^{2}%
+2ABd+B^{2}$, and hence $\left\langle G[\{v,w\}]\right\rangle =-A^{2}-A^{-2}$.

If $v$ has a neighbor $z$ other than $w$ then Theorem \ref{marked pivot} tells
us that $\left\langle G\right\rangle =\left\langle G_{c}^{vz}\right\rangle $.
In $G_{c}^{vz}$, $v$ and $w$ are adjacent, unmarked vertices, with $v$ looped
and $w$ unlooped. When constructing $G_{c}^{vz}$ we first construct $G^{vz}$,
by toggling all adjacencies between vertices $x,y\not \in \{v,z\}$ such that
$x$ is adjacent to $v$, $y$ is adjacent to $z$ and either $x$ is not adjacent
to $z$ or $y$ is not adjacent to $v$. This includes toggling every adjacency
between $w$ and a neighbor of $z$ other than $v$. As $w$ is adjacent to none
of these vertices in $G$, and not adjacent to any other vertex aside from $v$,
the neighbors of $w$ in $G^{vz}$ are the same as the neighbors of $z$. The
construction of $G_{c}^{vz} $ is completed by interchanging the neighbors of
$v$ and $z$, so the neighbors of $w$ in $G_{c}^{vz}$ are the same as the
neighbors of $v$ in $G_{c}^{vz}$. Proposition \ref{2movea} then tells us that
$\left\langle G_{c}^{vz}\right\rangle =\left\langle G_{c}^{vz}%
-v-w\right\rangle $.
\end{proof}

\begin{proposition}
\label{2movec} Suppose $G$ is a marked graph with two adjacent vertices $v$
and $w$ such that $v$ is looped and marked, $w$ is unlooped and unmarked, and
$v$ and $w$ have precisely the same neighbors outside $\{v,w\}$. If $v$ has no
neighbor other than $w$ then $\left\langle G\right\rangle =\left\langle
G^{+}-v-w\right\rangle $. Otherwise, $\left\langle G\right\rangle
=\left\langle G_{c}^{vz}-v-w\right\rangle $ for any $z\neq w$ that is adjacent
to $v$.
\end{proposition}

\begin{proof}
If $v$ has no neighbor other than $w$ the preceding proposition applies.

Suppose $z$ is a neighbor of $v$ other than $w$; then $\left\langle
G\right\rangle =\left\langle G_{c}^{vz}\right\rangle $. In $G^{vz}$, the
adjacency between $w$ and every vertex $x\notin\{v,z\}$ that is adjacent to
only one of $v,z$ is toggled; as $w$ has the same neighbors outside $\{v,w\}$
as $v$ has, it follows that $w$ has the same neighbors in $G^{vz}$ as $z$ has.
The neighbors of $v$ and $z$ in $G^{vz}$ are exchanged in $G_{c}^{vz}$, so
Proposition \ref{2movea} applies to $v$ and $w$ in $G_{c}^{vz}$.
\end{proof}

\begin{definition}
A \emph{marked-graph}\textit{\ }$\Omega.2$ \emph{Reidemeister move} is
constructed as follows. Any finite sequence of marked pivots may be applied to
a marked graph $H$; let $G$ denote the result. $G$ is replaced with $G-v-w $,
$G^{+}-v-w$ or $G_{c}^{vz}-v-w$ as in one of the three propositions above. Any
finite sequence of marked pivots may then be applied. The inverse of an
$\Omega.2$ move is also an $\Omega.2$ move, as is the move obtained by first
toggling all loops in $H$, and then toggling all loops in the final graph.
\end{definition}

\bigskip

Proposition \ref{looptoggle} and Theorem \ref{marked pivot} tell us that if
$G^{\prime}$ is obtained from $G$ by using a marked-graph $\Omega.2$ move to
adjoin or remove two vertices, then $\left\langle G\right\rangle =\left\langle
G^{\prime}\right\rangle $.

\bigskip

Recall that if $D$ is a link diagram then an\textit{\ }$\Omega.2$ Reidemeister
move on $D$ involves replacing one of the two configurations pictured on the
left or right of Figure \ref{lbracfi7} with the configuration in the middle.%

\begin{figure}
[ptb]
\begin{center}
\includegraphics[
trim=0.937041in 7.491790in 0.000000in 1.072090in,
height=1.6276in,
width=5.5106in
]%
{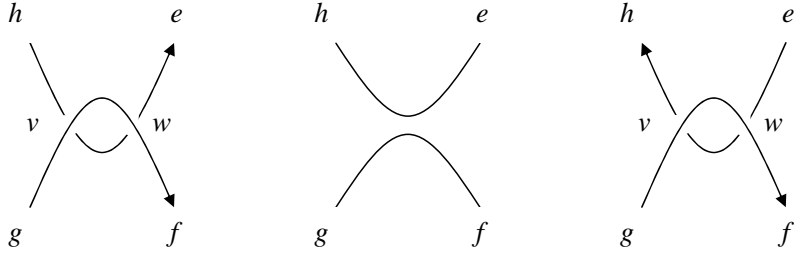}%
\caption{An\textit{\ }$\Omega.2$ Reidemeister move is applied to a link
diagram by replacing the portion pictured on the left or right (or its mirror
image) with the portion pictured in the middle.}%
\label{lbracfi7}%
\end{center}
\end{figure}

\bigskip

\begin{proposition}
\label{2move} If $D^{\prime}$ is obtained by applying an\textit{\ }$\Omega.2$
Reidemeister move to a link diagram $D$ then there are Euler systems $C$ and
$C^{\prime}$ such that $\mathcal{L}(D^{\prime},C^{\prime})$ is obtained from
$\mathcal{L}(D,C)$ by applying a marked-graph $\Omega.2$ move.
\end{proposition}

\begin{proof}
Suppose first that $D^{\prime}$ is obtained from $D$ by replacing the
configuration on the left in Figure \ref{lbracfi7} with the configuration that
appears in the middle, and $C$ is an Euler system for the directed universe
$\vec{U}$ of $D$. There are four possible arrangements of marks: either of
$v,w$ could be marked, or neither, or both. In every one of these
configurations the vertices of $U$ corresponding to $v$ and $w$ are interlaced
with respect to $C$, so if $w$ is marked we may replace $C$ with $C\ast v\ast
w\ast v$. We proceed under the assumption that $w$ is not marked.

If neither $v$ nor $w$ is marked, the Euler system $C$ must contain a circuit
that leaves $w$ along the edge corresponding to the pictured arc $e$, and
whose next appearance in the pictured portion of $U$ occurs along the edge
corresponding to the pictured arc $g$. Then the crossings $v$ and $w$
correspond to adjacent, unmarked vertices $v,w\in V(U)$ such that such that
$v$ is looped, $w$ is unlooped, and $v$ and $w$ are adjacent to the same
vertices outside $\{v,w\}$. Removing $v$ and $w$ from $\mathcal{L}(D,C)$ is an
instance of Proposition \ref{2movea}, and results in $\mathcal{L}(D^{\prime
},C^{\prime})$, where $C^{\prime}$ is the Euler system for $D^{\prime}$
obtained in the obvious way from $C$.

If $v$ is marked and $w$ is not, $C$ must instead contain a circuit $\gamma$
that leaves the portion of $D$ pictured in Figure \ref{lbracfi7} along $e$ and
re-enters along $h$. The crossings $v$ and $w$ again correspond to adjacent
vertices $v,w\in V(U)$ such that $v$ is looped, $w$ is unlooped, and $v$ and
$w$ are adjacent to the same vertices outside $\{v,w\}$. If $\gamma$ passes
through no crossing of $D$ only once on the way from $e$ to $h$, then $v$ and
$w$ have no neighbor in $\mathcal{L}(D,C)$ except each other. Removing $v$ and
$w$ is then an instance of the first clause of Proposition \ref{2moveb}, and
$\mathcal{L}(D,C)-v-w=\mathcal{L}(D^{\prime},C^{\prime})$ where $C^{\prime}$
is the Euler system for $D^{\prime}$ that contains the $e$-to-$h$ and
$f$-to-$g$ portions of $\gamma$ as separate circuits. If $\gamma$ passes
through a crossing $z$ only once on the way from $e$ to $h$, then Lemma
\ref{transform} tells us that $\mathcal{L}(D,C)_{c}^{vz}=\mathcal{L}(D,C\ast
v\ast z\ast v)$. As $v$ and $w$ are both unmarked in $\mathcal{L}%
(D,C)_{c}^{vz}$, we may refer to the preceding paragraph.

Suppose now that $D^{\prime}$ is obtained from $D$ by replacing the
configuration on the right in Figure \ref{lbracfi7} with the configuration in
the middle. It is impossible for both $v$ and $w$ to be marked, for this would
result in an ``extra'' circuit incident on only $v$ and $w$; as $C$ is an
Euler system, it cannot contain such a circuit.

If neither $v$ nor $w$ is marked, $C$ contains a circuit that leaves $w$ along
the edge corresponding to $f$ and then re-enters the portion of $D$ pictured
in Figure \ref{lbracfi7} at $e$. Then $v$ is looped, $w$ is unlooped, and
$v,w$ are nonadjacent vertices with the same neighbors. Removing $v$ and $w$
from $\mathcal{L}(D,C)$ is an instance of Proposition \ref{2movea}, and the
result is $\mathcal{L}(D^{\prime},C^{\prime})$ where $C^{\prime}$ is the Euler
system obtained from $C$ in the obvious way.

If $v$ is marked and $w$ is not then $v$ is the only neighbor of $w$ in
$\mathcal{L}(D,C)$. As in the third paragraph, we have two cases. If $w$ is
the only neighbor of $v$ in $\mathcal{L}(D,C)$ then we refer to the first
clause of Proposition \ref{2moveb}. $\mathcal{L}(D,C)-v-w=\mathcal{L}%
(D^{\prime},C^{\prime})$ where $C^{\prime}$ is obtained from $C$ by replacing
the circuit $\gamma$ that passes through $v$ and $w$ with two separate
circuits, one including the $h$-to-$e$ portion of $\gamma$ and the other
including the $f$-to-$g$ portion of $\gamma$. If instead $v$ has a neighbor
$z\neq w$ in $\mathcal{L}(D,C)$, then we replace $C$ by $C\ast v\ast z\ast v$
and refer to the preceding paragraph.

If $w$ is marked and $v$ is not then we toggle all loops, apply the preceding
paragraph with the roles of $v$ and $w$ reversed, and then toggle all loops again.
\end{proof}

\section{$\Omega.3$ moves and the reduced bracket}

\begin{proposition}
\label{3movea} Suppose $G$ is a marked graph with three unmarked vertices
$u,v,w$ such that $u,v,w$ are all adjacent to each other, $u$ is looped, $v$
and $w$ are unlooped, and every vertex $x\notin\{u,v,w\}$ is adjacent to
either 0 or precisely two of $u,v,w$. Let $G^{\prime}$ be the graph obtained
from $G$ by removing all three edges $\{u,v\}$, $\{u,w\}$ and $\{v,w\}$. Then
$\left\langle G\right\rangle =\left\langle G^{\prime}\right\rangle $.
\end{proposition}

\begin{proof}
Each $T\subseteq V(G)-\{u,v,w\}$ corresponds to eight subsets of $V(G)$, which
make eight contributions to $\left\langle G\right\rangle $ and eight
contributions to $\left\langle G^{\prime}\right\rangle $. In each case, three
of the eight cancel. The proposition is true because the remaining five happen
to equal each other.

Considering $G$ first, observe that the $GF(2)$-nullities of $\mathcal{A}%
(G)_{T}$, $\mathcal{A}(G)_{T\cup\{u\}}$ and $\mathcal{A}(G)_{T\cup\{u,v\}}$
are (respectively)
\begin{align*}
\nu\left(
\begin{array}
[c]{cccc}%
1 & 1 & 1 & \rho_{1}\\
1 & 0 & 1 & \rho_{2}\\
1 & 1 & 0 & \rho_{1}+\rho_{2}\\
\kappa_{1} & \kappa_{2} & \kappa_{1}+\kappa_{2} & M
\end{array}
\right)   & =\nu\left(
\begin{array}
[c]{cccc}%
1 & 0 & 0 & \mathbf{0}\\
0 & 1 & 0 & \rho_{1}+\rho_{2}\\
0 & 0 & 1 & \rho_{1}\\
\mathbf{0} & \kappa_{1} & \kappa_{1}+\kappa_{2} & M
\end{array}
\right)  ,\\
\nu\left(
\begin{array}
[c]{cccc}%
0 & 1 & 1 & \rho_{1}\\
1 & 0 & 1 & \rho_{2}\\
1 & 1 & 0 & \rho_{1}+\rho_{2}\\
\kappa_{1} & \kappa_{2} & \kappa_{1}+\kappa_{2} & M
\end{array}
\right)   & =\nu\left(
\begin{array}
[c]{cccc}%
0 & 0 & 0 & \mathbf{0}\\
0 & 1 & 0 & \rho_{1}+\rho_{2}\\
0 & 0 & 1 & \rho_{1}\\
\mathbf{0} & \kappa_{1} & \kappa_{1}+\kappa_{2} & M
\end{array}
\right)
\end{align*}
and
\[
\nu\left(
\begin{array}
[c]{cccc}%
0 & 1 & 1 & \rho_{1}\\
1 & 1 & 1 & \rho_{2}\\
1 & 1 & 0 & \rho_{1}+\rho_{2}\\
\kappa_{1} & \kappa_{2} & \kappa_{1}+\kappa_{2} & M
\end{array}
\right)  =\nu\left(
\begin{array}
[c]{cccc}%
1 & 0 & 0 & \mathbf{0}\\
0 & 1 & 0 & \rho_{1}+\rho_{2}\\
0 & 0 & 1 & \rho_{1}\\
\mathbf{0} & \kappa_{1} & \kappa_{1}+\kappa_{2} & M
\end{array}
\right)
\]
for appropriate row vectors $\rho_{i}$, column vectors $\kappa_{i}$ and a
submatrix $M$. Consequently $\nu(\mathcal{A}(G)_{T})=\nu(\mathcal{A}%
(G)_{T\cup\{u,v\}})=\nu(\mathcal{A}(G)_{T\cup\{u\}})-1$; it follows that the
contributions of $T$, $T\cup\{u\}$ and $T\cup\{u,v\}$ to $\left\langle
G\right\rangle $ cancel each other.

Similarly, the contributions of $T$, $T\cup\{u\}$ and $T\cup\{u,v\}$ to
$\left\langle G^{\prime}\right\rangle $ all cancel each other.

The $GF(2)$-nullities of $\mathcal{A}(G)_{T\cup\{v\}}$ and $\mathcal{A}%
(G^{\prime})_{T\cup\{v\}}$ are
\begin{align*}
\nu\left(
\begin{array}
[c]{cccc}%
1 & 1 & 1 & \rho_{1}\\
1 & 1 & 1 & \rho_{2}\\
1 & 1 & 0 & \rho_{1}+\rho_{2}\\
\kappa_{1} & \kappa_{2} & \kappa_{1}+\kappa_{2} & M
\end{array}
\right)   & =\nu\left(
\begin{array}
[c]{cccc}%
1 & 0 & 0 & \mathbf{0}\\
0 & 0 & 0 & \rho_{1}+\rho_{2}\\
0 & 0 & 1 & \mathbf{0}\\
\mathbf{0} & \kappa_{1}+\kappa_{2} & 0 & M
\end{array}
\right)  \text{ \ and}\\
\nu\left(
\begin{array}
[c]{cccc}%
1 & 0 & 0 & \rho_{1}\\
0 & 1 & 0 & \rho_{2}\\
0 & 0 & 0 & \rho_{1}+\rho_{2}\\
\kappa_{1} & \kappa_{2} & \kappa_{1}+\kappa_{2} & M
\end{array}
\right)   & =\nu\left(
\begin{array}
[c]{cccc}%
0 & 1 & 0 & \mathbf{0}\\
1 & 0 & 0 & \mathbf{0}\\
0 & 0 & 0 & \rho_{1}+\rho_{2}\\
\mathbf{0} & \mathbf{0} & \kappa_{1}+\kappa_{2} & M
\end{array}
\right)
\end{align*}
respectively, so the contributions of $T\cup\{v\}$ to $\left\langle
G\right\rangle $ and $\left\langle G^{\prime}\right\rangle $ are equal.
Similarly, $T\cup\{w\}$ makes equal contributions to $\left\langle
G\right\rangle $ and $\left\langle G^{\prime}\right\rangle $, and so does
$T\cup\{u,v,w\}$ .

Also, the $GF(2)$-nullities of $\mathcal{A}(G)_{T\cup\{u,w\}}$ and
$\mathcal{A}(G^{\prime})_{T\cup\{v,w\}}$ are
\begin{align*}
\nu\left(
\begin{array}
[c]{cccc}%
0 & 1 & 1 & \rho_{1}\\
1 & 0 & 1 & \rho_{2}\\
1 & 1 & 1 & \rho_{1}+\rho_{2}\\
\kappa_{1} & \kappa_{2} & \kappa_{1}+\kappa_{2} & M
\end{array}
\right)   & =\nu\left(
\begin{array}
[c]{cccc}%
0 & 1 & 0 & \rho_{1}\\
1 & 0 & 0 & \rho_{2}\\
0 & 0 & 1 & \mathbf{0}\\
\kappa_{1} & \kappa_{2} & \mathbf{0} & M
\end{array}
\right)  \text{ \ and}\\
\nu\left(
\begin{array}
[c]{cccc}%
1 & 0 & 0 & \rho_{1}\\
0 & 1 & 0 & \rho_{2}\\
0 & 0 & 1 & \rho_{1}+\rho_{2}\\
\kappa_{1} & \kappa_{2} & \kappa_{1}+\kappa_{2} & M
\end{array}
\right)   & =\nu\left(
\begin{array}
[c]{cccc}%
0 & 0 & 1 & \rho_{1}\\
0 & 1 & 0 & \rho_{2}\\
1 & 0 & 0 & \mathbf{0}\\
\mathbf{0} & \kappa_{1} & \kappa_{2} & M
\end{array}
\right)
\end{align*}
respectively, so the contributions of $T\cup\{u,w\}$ to $\left\langle
G\right\rangle $ and $T\cup\{v,w\}$ to $\left\langle G^{\prime}\right\rangle $
are equal. Similarly, the contributions of $T\cup\{u,w\}$ to $\left\langle
G^{\prime}\right\rangle $ and $T\cup\{v,w\}$ to $\left\langle G\right\rangle $
are equal.
\end{proof}

\bigskip

\begin{definition}
\label{3move} A \emph{marked-graph}\textit{\ }$\Omega.3$ \emph{Reidemeister
move} is obtained by composing one of the moves described in Proposition
\ref{3movea} with any finite sequence of marked pivots and $\Omega.2$
Reidemeister moves. The inverse of an $\Omega.3$ move is also an $\Omega.3$
move, as is the move obtained by first toggling all loops in the original
graph, and then toggling all loops in the final graph.
\end{definition}

\bigskip

Theorem \ref{marked pivot}, Proposition \ref{3movea} and the results of
Section 8 imply that $\Omega.3$ moves preserve the reduced marked-graph bracket.

\bigskip

Including composition with $\Omega.2$ moves in Definition \ref{3move} is
suggested by \"{O}st-\allowbreak lund's observation that different-looking
$\Omega.3$ moves are interrelated through composition with $\Omega.2$ moves
\cite{O}. The unmarked $\Omega.3$ moves used in \cite{TZ} (see Figure
\ref{lbracf8a}) are interrelated in the same way, but composition with
$\Omega.2 $ moves was not useful in \cite{TZ} because the proof of the
invariance of the reduced graph bracket under $\Omega.3$ moves used induction
on the number of vertices, and $\Omega.2$ moves may increase the number of
vertices. Moreover the inductive argument used to verify invariance under one
$\Omega.3$ move actually required invariance under different $\Omega.3$ moves
on smaller graphs. The arguments we present here do not involve induction on
the number of vertices, so \"{O}stlund's observation is quite useful.%

\begin{figure}
[t]
\begin{center}
\includegraphics[
trim=1.081051in 4.169652in 0.000000in 0.774455in,
height=3.7671in,
width=4.6847in
]%
{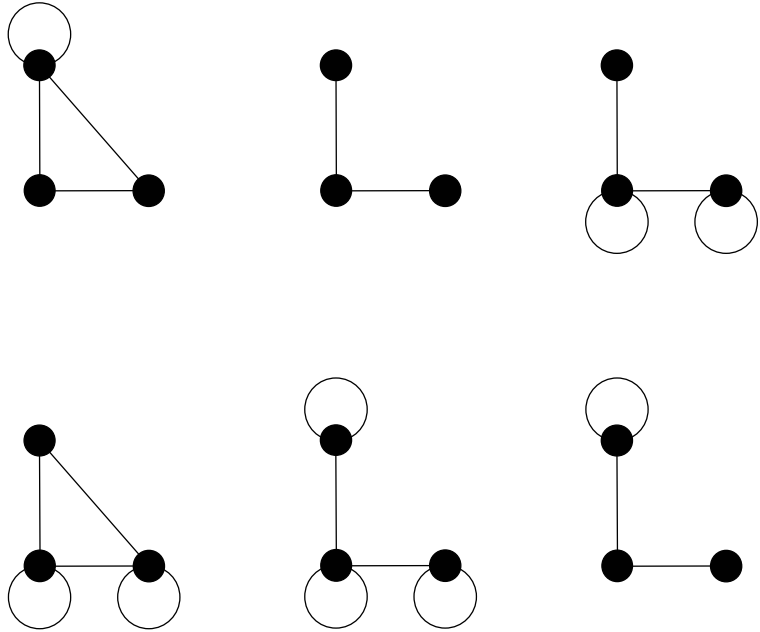}%
\caption{The examples of $\Omega.3$ moves mentioned in Corollary \ref{3moveb}
involve toggling\ the non-loop edges in one of the six pictured
configurations; every vertex outside the picture must have either 0 or
precisely 2 neighbors among the three unmarked vertices that are pictured.}%
\label{lbracf8a}%
\end{center}
\end{figure}

\begin{corollary}
\label{3moveb} Suppose $G$ is a marked graph with three unmarked vertices
$u,v,w$ such that the induced subgraph $G[\{u,v,w\}]$ is isomorphic to one of
the six 3-vertex graphs pictured in Figure \ref{lbracf8a}, and every vertex
$x\notin\{u,v,w\}$ is adjacent to either 0 or precisely two of $u,v,w$. Then
there is a marked-graph $\Omega.3$ Reidemeister move whose effect is to toggle
all the non-loop edges among $u,v,w$.
\end{corollary}

\begin{proof}
The instance of the proposition involving the 3-vertex graph that appears at
the top left of Figure \ref{lbracf8a} is verified in Proposition \ref{3movea}.
As was observed in \cite{TZ}, the other two instances that appear in the top
row of Figure \ref{lbracf8a} can be obtained from this one through composition
with $\Omega.2$ moves. (The same constructions given in \cite{TZ} apply here,
because the $\Omega.2$ moves of Proposition \ref{2movea} coincide locally with
the unmarked $\Omega.2$ moves of \cite{TZ}, and do not require any special
assumptions regarding the marked vertices of $G$.) The three instances
pictured in the bottom row of Figure \ref{lbracf8a} follow, for if $G$
contains such an induced subgraph then we may toggle all loops in $G$, apply
the corresponding move pictured in the top row of Figure \ref{lbracf8a}, and
then toggle all loops in the resulting graph.
\end{proof}

\bigskip

N.B. Corollary \ref{3moveb} states only that the moves illustrated in Figure
\ref{lbracf8a} are examples of marked-graph $\Omega.3$ moves. As Definition
\ref{3move} includes composition with $\Omega.2$ moves and marked pivots, it
is not possible to give an exhaustive list of $\Omega.3$ moves.

\bigskip

If $D$ is a link diagram then an\textit{\ }$\Omega.3$ Reidemeister move
involves replacing a configuration like one of those in Figure \ref{lbracfi9}
with a configuration like the other. The arcs appearing in the figure may be
oriented in any way, and one may replace both configurations by their mirror images.%

\begin{figure}
[h]
\begin{center}
\includegraphics[
trim=0.936216in 6.836980in 0.000000in 0.774644in,
height=2.3436in,
width=5.5132in
]%
{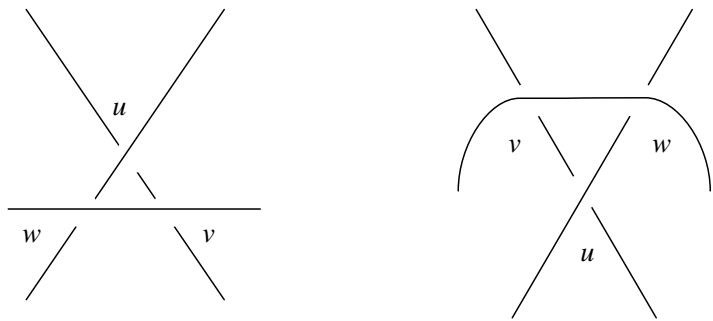}%
\caption{An $\Omega.3$ move on a link diagram.}%
\label{lbracfi9}%
\end{center}
\end{figure}

\begin{proposition}
If $D^{\prime}$ is obtained by applying an\textit{\ }$\Omega.3$ Reidemeister
move to a link diagram $D$ then there are Euler systems $C$ and $C^{\prime}$
such that $\mathcal{L}(D^{\prime},C^{\prime})$ is obtained from $\mathcal{L}%
(D,C)$ by applying a marked-graph $\Omega.3$ move or a finite (possibly empty)
sequence of marked pivots.
\end{proposition}

\begin{proof}
Let $C$ be an Euler system for the directed universe $\vec{U}$ of the diagram
$D$ which contains the left-hand side of Figure \ref{lbracfi9}, and let
$D^{\prime}$ be the link diagram that differs from $D$ only by containing the
right-hand side of Figure \ref{lbracfi9}.

Suppose first that the crossings are all unmarked.

If the horizontal arc in the left-hand portion of Figure \ref{lbracfi9} is
oriented from left to right and the other two arcs are oriented from top to
bottom, then the fact that the circuit of $C$ that passes through this portion
of $D$ is an Euler circuit implies that it is either $uvC_{1}uwC_{2}vwC_{3}$
or $uvC_{1}vwC_{2}uwC_{3}$; the Euler circuit in the right-hand portion of
Figure \ref{lbracfi9} that arises naturally from $C$ will then be
$vuC_{1}wuC_{2}wvC_{3}$ or $vuC_{1}wvC_{2}wuC_{3}$, respectively. All three
crossings are positive, and each crossing outside the figure is interlaced
with either 0 or precisely 2 of these; for instance a crossing that appears
once in $C_{1}$ and once in $C_{2}$ is interlaced with $u$ and $w$ in the
first case, and $v$ and $w$ in the second case. Moreover each crossing in the
left-hand portion of the figure is interlaced with the same outside crossings
as the crossing on the right that has the same label ($u$, $v$ or $w$). It
follows that the pictured $\Omega.3$ move is the one that appears in the
middle of the top row of Figure \ref{lbracf8a}.

If Figure \ref{lbracfi9} is replaced by its mirror image and the orientations
are the same as in the preceding paragraph, the same argument holds with all
loops toggled.

According to Lemma 3 of \cite{O}, every other $\Omega.3$ move represented by
some oriented version of Figure \ref{lbracfi9} or its mirror image is a
composition of $\Omega.2$ moves and one of the two moves just discussed. The
result then follows from the two cases just discussed and Proposition
\ref{2move}.

Suppose now that the portion of $D$ pictured on the left-hand side of Figure
\ref{lbracfi9} has a single marked vertex $x\in\{u,v,w\}$. If $x$ is
interlaced with any crossing $y$ outside the pictured portion of $D$, then we
may replace $C$ with $C\ast x\ast y\ast x$, which has no marked vertices in
the pictured portion of $D$. It remains to consider the possibility that
precisely one of $u,v,w$ is marked and the marked element of $\{u,v,w\}$ is
not interlaced with any crossing outside $\{u,v,w\}$. Once again there are
several special cases, but we need only discuss one explicitly.

For instance, suppose that the left-hand side is oriented as in the first
paragraph of the proof, $v$ is marked and $u,w$ are not. The circuit of $C$
incident on the left-hand side of Figure \ref{lbracfi9} is $uvwC_{1}%
vC_{2}uwC_{3}$ or $vC_{1}uvwC_{3}uwC_{2}$; the obvious Euler system for the
right-hand side of Figure \ref{lbracfi9} then includes the incident circuit
$vC_{1}wvuC_{2}wuC_{3}$ or $vC_{3}wuC_{2}wvuC_{1}$, respectively. The
assumption that $v$ has no neighbor outside $\{u,w\}$ implies that no crossing
appears precisely once within $C_{1}$; it follows that $\mathcal{L}(D^{\prime
},C^{\prime})$ is actually identical to $\mathcal{L}(D,C)$. If only $w$ is
marked then the circuit of $C$ incident on the left-hand side of Figure
\ref{lbracfi9} is $uwC_{1}vwC_{2}uvC_{3}$ or $vwC_{1}uwC_{2}uvC_{3}$, and the
obvious Euler system for the right-hand side of Figure \ref{lbracfi9} includes
the incident circuit $wvC_{1}wuC_{2}vuC_{3}$ or $wuC_{1}wvC_{2}vuC_{3}$,
respectively. The assumption that $w$ has no neighbor outside $\{u,v\}$
implies that no crossing appears precisely once within $C_{1}$; it follows
again that $\mathcal{L}(D^{\prime},C^{\prime})$ is identical to $\mathcal{L}%
(D,C)$. Finally, if $u$ is the lone marked vertex in the left-hand portion of
Figure \ref{lbracfi9} then the incident circuit of $C$ is either
$uwC_{1}uvC_{2}vwC_{3}$ (in which $u$ is interlaced with $w$) or
$uvC_{1}uwC_{2}vwC_{3}$ (in which $u$ is interlaced with $v$). If $u$ is
interlaced with $v$ then $C\ast u\ast v\ast u$ will have $v$ marked instead of
$u$, and if $u$ is interlaced with $w$ then $C\ast u\ast w\ast u$ will have
$w$ marked instead of $u$; either way a transposition leads to one of the
first two cases.

As before, the mirror image is handled by toggling all loops, and other
oriented versions of the left-hand side of Figure \ref{lbracfi9} (or its
mirror image) with precisely one marked vertex are handled by expressing the
corresponding moves as compositions of $\Omega.2$ moves with one of the two
$\Omega.3$ moves already considered.

Suppose now that the portion of $D$ pictured on the left-hand side of Figure
\ref{lbracfi9} includes precisely two marked vertices. If a marked crossing
$x\in\{u,v,w\}$ is interlaced with any crossing $y$ outside the pictured
portion of $D$, then we may replace $C$ with $C\ast x\ast y\ast x$, which has
only one marked vertex in the pictured portion of $D$; hence we may assume no
marked crossing among $\{u,v,w\}$ is interlaced with any crossing outside
$\{u,v,w\}$. The hypothesis that every crossing outside $\{u,v,w\}$ is
interlaced with an even number of elements of $\{u,v,w\}$ then allows us to
assume that no crossing outside $\{u,v,w\}$ is interlaced with any of $u,v,w$.
Moreover, if the two marked crossings are interlaced then the marks are
removed by replacing $C$ with the Euler system obtained by a transposition on
these two vertices, so we may proceed assuming they are not interlaced.

Suppose again that the left-hand side of Figure \ref{lbracfi9} is oriented as
above. If $u,v$ are the two crossings that are marked then the incident Euler
circuit of $C$ must be $uwC_{1}uvwC_{2}vC_{3}$. There is then an Euler circuit
$vuC_{1}wvC_{2}wuC_{3}$ incident on the right-hand portion of Figure
\ref{lbracfi9}; let $C^{\prime}$ be the Euler system that coincides with $C$
except for this replacement of circuits incident on the pictured portions of
$D$ and $D^{\prime}$. The assumption that no outside crossing is interlaced
with any of $u,v,w$ with respect to $C$ tells us that every outside crossing
appears twice in one $C_{i}$; hence no outside crossing is interlaced with any
of $u,v,w$ with respect to $C^{\prime}$ either. Considering $uwC_{1}%
uvwC_{2}vC_{3}$, we see that $\mathcal{L}(D,C)$ has a connected component
whose only vertices are $u$, $v$ and $w$; $u$ and $v$ are marked, nonadjacent
neighbors of $w$. Considering $vuC_{1}wvC_{2}wuC_{3}$, we see that
$\mathcal{L}(D^{\prime},C^{\prime})$ has a connected component whose only
vertices are $u$, $v$ and $w$; $u$ and $w$ are marked, nonadjacent neighbors
of $v$. That is, $\mathcal{L}(D^{\prime},C^{\prime})=\mathcal{L}(D,C)_{c}%
^{vw}$. If $u$ and $w$ are the two that are marked then the Euler circuit of
$C$ incident of the left-hand portion of Figure \ref{lbracfi9} must be
$uvC_{1}uwC_{2}vwC_{3}$; there is then an Euler circuit $wvuC_{1}vC_{2}%
wuC_{3}$ incident on the right-hand portion. If $C^{\prime}$ is the Euler
system that differs from $C$ only by including this latter circuit then
$\mathcal{L}(D^{\prime},C^{\prime})=\mathcal{L}(D,C)_{c}^{uv}$. Finally, if
$v$ and $w$ are the two that are marked then the Euler circuit of $C$ incident
on the left-hand portion of Figure \ref{lbracfi9} must be $uwC_{1}%
vC_{2}uvwC_{3}$; hence an Euler circuit incident on the right-hand portion is
$wvC_{1}wuC_{2}vuC_{3}$. If $C^{\prime}$ is the Euler system that differs from
$C$ only by including this latter circuit then $\mathcal{L}(D^{\prime
},C^{\prime})=\mathcal{L}(D,C)_{c}^{uv}$.

Finally, suppose all three of $u,v,w$ are marked with respect to $C$. Any
adjacency involving two of them allows a marked pivot that will eliminate two
of the three marks, so we may assume no two of $u,v,w$ are interlaced with
respect to $C$. This is not possible.
\end{proof}

\bigskip

\begin{corollary}
Suppose $D$ and $D^{\prime}$ are diagrams of the same oriented virtual link
type, with directed universes $\vec{U}$ and $\vec{U}^{\prime}$. Let $C$ and
$C^{\prime}$ be directed Euler systems for $\vec{U}$ and $\vec{U}^{\prime}$.
Then $\mathcal{L}(D,C)$ and $\mathcal{L}(D^{\prime},C^{\prime})$ can be
transformed into each other through some finite sequence of marked pivots,
marked-graph $\Omega.2$ and $\Omega.3$ Reidemeister moves, and adjunctions and
deletions of isolated, unmarked vertices.
\end{corollary}

\begin{proof}
$D$ and $D^{\prime}$ can be transformed into each other through some finite
sequence of Reidemeister moves of types $\Omega.1$, $\Omega.2$ and $\Omega.3$
and virtual Reidemeister moves. $\Omega.1$ moves correspond to adjunctions and
deletions of isolated, unmarked vertices. Virtual Reidemeister moves have no
effect on $\mathcal{L}(D,C)$.
\end{proof}

\section{The Jones polynomial}

\begin{definition}
A \emph{Reidemeister move} on a marked graph is one of the following.

$\Omega.1.$ An isolated, unmarked vertex may be adjoined or deleted. The
vertex may be looped or unlooped.

$\Omega.2.$ One of the marked-graph $\Omega.2$ Reidemeister moves of Section 8
may be applied.

$\Omega.3.$ One of the marked-graph $\Omega.3$ Reidemeister moves of Section 9
may be applied.
\end{definition}

\begin{definition}
Let $G$ be a marked graph with $n$ vertices, $\ell$ of which are looped. Then
the \emph{Jones polynomial} of $G$ is
\[
V_{G}(t)=(-1)^{n}\cdot t^{(3n-6\ell)/4}\cdot\left\langle G\right\rangle
(t^{-1/4}).
\]

\end{definition}

If $G$ has no free loops and no marked vertices then the bracket, reduced
bracket and Jones polynomial of $G$ coincide with those discussed in
\cite{TZ}. Much of the discussion of \cite{TZ} is therefore still relevant. In
particular, several important aspects of the classical theory of the Jones
polynomial do not extend to this context; for instance, we cannot expect to
calculate $V_{G}(t)$ using a recursive algorithm that reduces $G$ to
\textquotedblleft ungraphs\textquotedblright\ in the same way that the Jones
polynomial of a classical link can be calculated by repeatedly using Theorem
12 of \cite{Jo} (the correct version of which is the formula $t^{-1}V_{L^{+}%
}-tV_{L^{-}}=(t^{1/2}-t^{-1/2})V_{L}$) to reduce $L$ to unlinks.

\bigskip

However the fundamental property that the Jones polynomial is a link type
invariant does survive the generalization to marked graphs.

\begin{theorem}
If $G$ is a marked graph and $G^{\prime}$ is obtained from $G$ through marked
pivots and Reidemeister moves then $V_{G}=V_{G^{\prime}}$.
\end{theorem}

\begin{proof}
A marked pivot or $\Omega.3$ move does not affect $n$, $\ell$ or $\left\langle
G\right\rangle $.

If $E_{1}$ consists of one isolated, unmarked, unlooped vertex then the
disjoint union $G\cup E_{1}$ has $[G\cup E_{1}]=(Ad+B)\cdot\lbrack G]$, so
$\left\langle G\cup E_{1}\right\rangle =(-A^{3})\cdot\left\langle
G\right\rangle $ and $V_{G\cup E_{1}}=-(-1)^{n}t^{3/4}\cdot t^{(3n-6\ell
)/4}\cdot(-t^{-3/4})\left\langle G\right\rangle (t^{-1/4})=V_{G}$, where $n$
and $\ell$ refer to $G$. A similar calculation shows that adjoining a single
isolated, unmarked, looped vertex does not change $V$.

An $\Omega.2$ move has no effect on $\left\langle G\right\rangle $. It adds
$2$ or $-2$ to the number of vertices while adding $1$ or $-1$ (resp.) to the
number of loops, so it also has no effect on $(-1)^{n}\cdot t^{(3n-6\ell)/4}$.
\end{proof}

\bigskip

In closing we observe that a special feature of the marked-graph bracket is
that both definitions (the nullity formula and the recursion) are
``local''\ in the sense that information from different connected components
of $G$ is processed separately. Consequently there are multivariable forms of
the marked-graph bracket and Jones polynomials, in which different connected
components of $G$ are associated with different sets of variables. For
classical knots, disconnected interlacement graphs arise from connected sums,
which are well understood \cite{Sc}; on the other hand, the connected sum of
virtual knot types is not uniquely defined \cite{KM, KS}. Whether there is any
useful application of the multivariable bracket and Jones polynomials remains
to be seen.

\end{document}